\documentclass[12pt]{article}

\usepackage{amsmath,enumerate,amsfonts,amssymb,color,graphicx,amsthm,stmaryrd}
\usepackage{hyperref}

\def\daspr{{\delta_*}}
\def\Caspr{{C_*}}
\def\daste{{\delta_{0}}}
\def\Caste{{C_{0}}}
\def\Cfive{{K}}

\def\RR{{\mathbb R}}

\def\Ric{{\rm Ric}}

\def\diag{{\rm diag}}

\begin{document}
\newtheorem{THM}{Theorem}
\renewcommand*{\theTHM}{\Alph{THM}}
\newtheorem{Def}{Definition}[section]
\newtheorem{thm}{Theorem}[section]
\newtheorem{lem}{Lemma}[section]
\newtheorem{rem}{Remark}[section]
\newtheorem{question}{Question}[section]
\newtheorem{prop}{Proposition}[section]
\newtheorem{cor}{Corollary}[section]
\newtheorem{clm}{Claim}[section]
\newtheorem{step}{Step}[section]
\newtheorem{sbsn}{Subsection}[section]
\newtheorem*{conj}{Conjecture}

\title
{Symmetry, quantitative Liouville theorems and analysis of large solutions of conformally invariant
fully nonlinear elliptic equations}
\author{YanYan Li \footnote{Department of Mathematics, Rutgers University, Hill Center, Busch Campus, 110 Frelinghuysen Road, Piscataway, NJ 08854, USA. Email: yyli@math.rutgers.edu.}~\footnote{Y.Y. Li is partially supported by NSF grant DMS-1501004.}~ and Luc Nguyen \footnote{Mathematical Institute and St Edmund Hall, University of Oxford, Andrew Wiles Building, Radcliffe Observatory Quarter, Woodstock Road, Oxford OX2 6GG, UK. Email: luc.nguyen@maths.ox.ac.uk.}}
\date{}

\maketitle

\begin{abstract}
We establish blow-up profiles for any blowing-up sequence of solutions of general conformally invariant
fully nonlinear elliptic equations on Euclidean domains. We prove that (i) the distance between blow-up points is bounded from below by a universal positive number, (ii) the solutions are very close to a single standard bubble in a universal positive distance around each blow-up point, and (iii) the heights of these bubbles are comparable by a universal factor. As an application of this result, we establish a quantitative Liouville theorem.
\end{abstract}

\tableofcontents
\section{Introduction}

The main goal of this
 paper is to give a
 fine analysis of  blow-up solutions of 
conformally invariant fully nonlinear second order elliptic equations.

Let $n \geq 3$ be an integer and
\begin{equation}
\Gamma\subset \RR^n\ \mbox{be an open convex symmetric cone with vertex at the origin}
\label{01}
\end{equation}
satisfying
\begin{equation}
\Gamma_n\subset
\Gamma\subset
\Gamma_1,
\label{02}
\end{equation}
where
\[
\Gamma_n 	:=\{\lambda\in \RR^n\ |\
\lambda_i>0\ \forall\ i\}
, \qquad \Gamma_1 :=\Big\{\lambda\in \RR^n\ |\
\sum_{i=1}^n\lambda_i>0\Big\}.
\]
We assume that
\begin{equation}
f\in C^1(\Gamma)\cap C^0(\overline \Gamma)\ \mbox{is
symmetric in
} \ \lambda_i,
\label{03}
\end{equation}
\begin{equation}
f>0, \ \  \frac{  \partial f}{ \partial \lambda_i}>0\ 
\mbox{in}\ \Gamma\ \forall\ i,\ \
f=0\ \mbox{on}\ \partial \Gamma,
\label{04}
\end{equation}
\begin{equation}
f(\lambda) < 1\qquad\forall\
\lambda\in \Gamma \text{ satisfying } \ \sum_{i=1}^n \lambda_i <\delta.
\label{FU-1a}
\end{equation}

In \eqref{01} and \eqref{03}, the symmetric property of $\Gamma$ and $f$ is understood in the sense that if $\lambda \in \Gamma$ and $\tilde \lambda$ is a permutation of $\lambda$, then $\tilde \lambda \in \Gamma$ and $f(\tilde\lambda) = f(\lambda)$. Also, throughout the paper, 
 \[
 \parbox{.8\textwidth}{whenever we write $f(\lambda)$, we 
 implicitly assume that $\lambda\in \overline \Gamma$.}
 \]

When $\Gamma \neq \Gamma_1$, \eqref{FU-1a} is a consequence of \eqref{03} and \eqref{04} (cf. \cite[Proposition B.1]{LiNgBocher}). However, this does not have to be the case when $\Gamma = \Gamma_1$, for example when
\[
f(\lambda) = \Big(\sum_{i=1}^n \lambda_i\Big)^{\frac{n-1}{n+3}} \Big(\sum_{i=1}^n \lambda_i^2\Big)^{\frac{2}{n+3}}.
\]

Illuminating examples of $(f, \Gamma)$ are $(f, \Gamma)=(\sigma_k^{\frac 1k}, \Gamma_k)$ where $\sigma_k(\lambda) = \sum \lambda_{i_1} \ldots \lambda_{i_k}$ is the $k$-th elementary symmetric function and 
\begin{align*}
\Gamma_k
	&= \text{ the connected component of $\{\lambda \in \RR^n: \sigma_k(\lambda) > 0\}$ containing $\Gamma_n$}\\
	&= \{\lambda \in \RR^n: \sigma_l(\lambda) > 0 \text{ for all } 1 \leq l \leq k\}.
\end{align*}
Besides \eqref{01}-\eqref{FU-1a}, $(\sigma_k^{\frac 1k}, \Gamma_k)$ enjoys other nice and helpful properties, such as concavity and homogeneity properties of $\sigma_k^{1/k}$, Newton's inequalities, divergence and variational structures, etc., which we do not assume in this paper. In particular, we would like to note that no concavity or homogeneity assumption on $f$ is being made in the present paper.

For a positive $C^2$ function $u$, let $A^u$ be the $n\times n$ matrix with entries
\[
(A^u)_{ij} = - \frac{2}{n-2} u^{-\frac{n+2}{n-2}} \nabla_i \nabla_j u + \frac{2n}{(n-2)^2} u^{-\frac{2n}{n-2}}\,\nabla_i u \,\nabla_j u -  \frac{2}{(n-2)^2} u^{-\frac{2n}{n-2}}\,|\nabla u|^2\,\delta_{ij}.
\]
This is sometimes referred to as the conformal Hessian of $u$.

The conformal Hessian $A^u$ arises naturally in conformal geometry as follows. Recall that the Riemann curvature $Riem_g$ of a Riemannian metric $g$ can be decomposed into traced and traceless parts as
\[
Riem_g = A_g \owedge  g + W_g,
\]
where $A_g = \frac{1}{n-2}(\Ric_g - \frac{1}{2(n-1)}R_g\,g)$, $\Ric_g$, $R_g$ and $W_g$ are the Schouten curvature, the Ricci curvature, the scalar curvature and the Weyl curvature of $g$ and $\owedge$ denotes the Kulkarni-Nomizu product. While the $(1,3)$-valent Weyl curvature remains unchanged under a conformal change of the metric, the Schouten curvature is adjusted by a second order operator of the conformal factor. In particular, if we consider the metric $g_u :=  u^{\frac{4}{n-2}}g_{flat}$ conformal to the flat metric $g_{flat}$ on $\RR^n$, then the Schouten curvature $A_{g_u}$ of $g_u$ is given by the conformal Hessian in the form
\[
A_{g_u} = u^{\frac{4}{n-2}}\,(A^u)_{ij}\,dx^i\,dx^j.
\]
Consequently, we have
\[
\lambda(A_{g_u}) = \lambda(A^u)
\]
where $\lambda(A_{g_u})$ denotes the eigenvalues of $A_{g_u}$ with respect to the metric $g_u$ and $\lambda(A^u)$ denotes those of the matrix $A^u$.

$A^u$ enjoys a conformal invariance property, inherited from the conformal structure of $\RR^n$, which will be of special importance in our treatment. Recall that a map $\varphi: \RR^n \cup\{\infty\} \rightarrow \RR^n \cup\{\infty\}$ is called a M\"obius transformation if it is the composition of finitely many of the following types of transformations:
\begin{itemize}
\item a translation: $x \mapsto x + \bar x$ where $\bar x$ is a given vector in $\RR^n$,
\item a dilation: $x \mapsto a\,x$ where $a$ is a given positive scalar,
\item a Kelvin transformation: $x \mapsto \frac{x}{|x|^2}$.
\end{itemize}
For a function $u$ and a M\"obius transformation $\varphi$, let
\begin{equation}
u_\varphi = |J_\varphi|^{\frac{n-2}{2n}}u \circ \varphi,
	\label{Eq:uvarphi}
\end{equation}
where $J_\varphi$ is the Jacobian of $\varphi$. A calculation gives 
\[
A^{u_\varphi}(x) = O_\varphi(x)^t A^u(\varphi(x)) O_\varphi(x)
\]
for some orthogonal $n \times n$ matrix $O_\varphi(x)$. In particular,
\begin{equation}
\lambda(A^{u_\varphi}(x)) = \lambda(A^u( \varphi(x))).
	\label{Eq:CIProp}
\end{equation}

The main result of this paper concerns an analysis 
on the behavior of a sequence  $\{u_k\}\in C^2(B_3(0))$ satisfying
\begin{equation}
f(\lambda(A^{u_k})) =1, \
u_k>0,\
\mbox{in}\ B_3(0),
\label{eqB2}
\end{equation}
and
\begin{equation}
\sup_{B_1(0)}u_k\to \infty,
\label{eqB1}
\end{equation}
where
$(f, \Gamma)$ satisfies \eqref{01}-\eqref{FU-1a}. Note that no other assumptions on $u_k$ is made.

As is known, equation \eqref{eqB2} is 
a fully nonlinear
 elliptic equation. Fully nonlinear elliptic equations involving $f(\lambda(\nabla^2 u))$ were investigated in the classic paper of Caffarelli, Nirenberg and Spruck \cite{C-N-S-Acta}. 
 
 Our paper appears to be the first fine blow-up analysis in this fully nonlinear context. We expect this to serve as a crucial step in the study of the problem on Riemannian manifolds.

To obtain our result on fine analysis of  blow-up solutions,
we make use of the following Liouville theorems.

\begin{THM}[\cite{LiLi05}]\label{TheoremA}
 Let $(f, \Gamma)$ satisfy \eqref{01}-\eqref{04}
and
let $0<v\in C^2(\RR^n)$ satisfy 
\begin{equation}
f(\lambda(A^v))=1\qquad \mbox{in}\ \RR^n.
\label{eqB0}
\end{equation}
Then 
\begin{equation}
v(x)\equiv \left( \frac a{  1+b^2|x-\bar x|^2}  \right) ^{ \frac {n-2}2 },
\qquad x\in \RR^n
\label{vform}
\end{equation}
for some $\bar x\in \RR^n$ and some positive constants
$a$ and $b$ satisfying 
\[
f(2b^2a^{-2}, \cdots, 2b^2a^{-2})=1.
\]
\end{THM}

\begin{THM}[\cite{Li09-CPAM}]\label{TheoremADeg}
 Let $\Gamma$ satisfy \eqref{01} and \eqref{02},
and
let $0< v\in C^{0,1}_{loc}(\RR^n \setminus \{0\})$ satisfy 
\begin{equation}
\lambda(A^v) \in \partial \Gamma \mbox{ on }\ \RR^n \setminus \{0\}
\label{eqB0deg}
\end{equation}
in the viscosity sense (see Definition \ref{Def:DegVisSol} below). Then $v$ is radially symmetric about the origin and $v(r)$ is non-increasing in $r$.
\end{THM}

For $(f, \Gamma)=(\frac{1}{2n}\sigma_1, \Gamma_1)$, equation \eqref{eqB2} is
the critical exponent equation $-\Delta u= n(n-2)u^{ (n+2) /(n-2) }$ and
Theorem \ref{TheoremA} was proved by Caffarelli, Gidas and Spruck \cite{CGS}.
See also Gidas, Ni and Nirenberg \cite{G-N-N-1981} under some decay assumption of $u$ at infinity.
For $(f, \Gamma)=(\sigma_2^{1/2}, \Gamma_2)$ in $\RR^4$ and $v \in C^{1,1}_{loc}(\RR^4)$, 
the result was proved by Chang, Gursky and Yang \cite{CGY03-IP}.
 
\medskip
In fact we need a stronger version of Theorem \ref{TheoremA} (see Theorem \ref{proposition1}) and a variant of Theorem \ref{TheoremADeg} (see Theorem \ref{proposition1-deg}). For simplicity, readers are advised that in the main body of the paper 
\[
\parbox{.8\textwidth}{all theorems, propositions and lemmas hold under \eqref{01}-\eqref{FU-1a}, instead of the stated weaker hypotheses on $(f,\Gamma)$.}
\]

\begin{thm} Let $(f, \Gamma)$ satisfy 
\begin{align}
&\Gamma\subset \RR^n\ \mbox{be an open symmetric set},
\label{01weak}\\
&\Gamma\subset
\Gamma_1 \text{ and } \Gamma \cap \{\lambda + t\mu:  t > 0\} \text{ is convex for all $\lambda \in \Gamma, \mu \in \Gamma_n$},
\label{02weakX}\\
&f \in C^1(\Gamma) \text{ is symmetric in } \lambda_i \text{ and } \frac{  \partial f}{ \partial \lambda_i}>0\ 
\mbox{in}\ \Gamma\ \forall\ i.\label{0304weak}
\end{align}
 Assume that
  $0<v\in C^0(\RR^n)$,
$0<v_k\in C^2(B_{R_k}(0))$,
 $R_k\to\infty$,
\begin{equation}
f(\lambda(A^{v_k}))=1
\quad \ \ \mbox{in}\ B_{R_k}(0),
\label{ve}
\end{equation}
and 
\begin{equation}
v_k\to v\ \mbox{in}\ C^0_{loc}(\RR^n).
\label{1new}
\end{equation}
Then either $v$ is constant or $v$ is of the form
\eqref{vform}
for some $\bar x\in \RR^n$ and some positive constants
$a$ and $b$.

If it holds in addition that
\begin{align}
&\text{there exists $t_0 > 0$ such that $f(t_0, \ldots, t_0) < 1$,}
\label{Eq:ActaSIdiag}
\end{align}
then $v$ cannot be constant. If it holds further that
\begin{equation}
\Gamma + \Gamma_n = \{\lambda + \mu: \lambda \in \Gamma, \mu \in \Gamma_n\} \subset \Gamma,
	\label{02weakY}
\end{equation}
then the constants $a$ and $b$ in \eqref{vform} satisfy $(2b^2a^{-2}, \cdots, 2b^2a^{-2}) \in \Gamma$ and $$f(2b^2a^{-2}, \cdots, 2b^2a^{-2})=1.$$
\label{proposition1}
\end{thm}

\begin{rem}
In Theorem \ref{proposition1}, if condition \eqref{Eq:ActaSIdiag} is dropped, the case that $v$ is constant can occur. See the counterexample in Remark \ref{Rem:ASD}.
\end{rem}

\begin{thm} Let $(f, \Gamma)$ satisfy \eqref{01weak}-\eqref{0304weak}. Assume that
  $v_*\in C^0(\RR^n \setminus \{0\})$,
$0<v_k\in C^2(B_{R_k}(0) \setminus \{0\})$,
 $R_k\to\infty$,
\[
f(\lambda(A^{v_k}))=1
\quad \ \ \mbox{in}\ B_{R_k}(0) \setminus \{0\},
\]
and, for some $M_k \rightarrow \infty$,
\begin{equation}
M_k\,v_k\to v_*\ \mbox{in}\ C^0_{loc}(\RR^n \setminus \{0\}).
\label{1newdeg}
\end{equation}
Then $v_*$ is radially symmetric about the origin, i.e. $v_*(x) = v_*(|x|)$. In particular, if $v_k \in C^2(B_{R_k}(0))$, $f(\lambda(A^{v_k}))=1$ in $B_{R_k}(0)$ and $M_kv_k$ converges to $v_*$ in $C^0_{loc}(\RR^n)$, then $v_*$ is constant.
\label{proposition1-deg}
\end{thm}

\begin{rem}
When $(f,\Gamma)$ satisfies \eqref{01}-\eqref{04} and an additional hypothesis that $f$ is homogeneous of positive degree, the function $v_*$ in Theorem \ref{proposition1-deg} is a viscosity solution of \eqref{eqB0deg} and the conclusion follows from Theorem \ref{TheoremADeg}. However, when $f$ is not homogeneous, $v_*$ is not necessarily a viscosity solution of \eqref{eqB0deg}.
\end{rem}

It is not difficult to see that, under \eqref{01}-\eqref{04}, the function $v$ in Theorem \ref{proposition1} is a viscosity solution of \eqref{eqB0} (see Remark \ref{Cor:VCConf}). We have the following conjecture.

\begin{conj}
Let  $(f, \Gamma)$ satisfy \eqref{01}-\eqref{04},
and let  $0<v\in C^0_{loc}(\RR^n)$
be a viscosity solution of  \eqref{eqB0}.
Then $v$ is of the form \eqref{vform}
for some $\bar x\in \RR^n$ and some positive constants
$a$ and $b$.
\end{conj}

The notion of 
viscosity solutions given below is consistent with that in \cite{Li09-CPAM}.

\begin{Def}\label{Def:fVisSol} A positive continuous function
 $v$ in an open set $\Omega\subset \RR^n$
is a viscosity supersolution (respectively, subsolution) of 
$$
f(\lambda(A^v))=1, \ \ \mbox{in}\ \Omega,
$$
when the following holds: if $x_0\in \Omega$,
$\varphi\in C^2(\Omega)$, $(v-\varphi)(x_0)=0$,
and $v-\varphi\ge 0$ near $x_0$, then
$$
f(\lambda(A^\varphi(x_0))\ge 1
$$
(respectively, if $(v-\varphi)(x_0)=0$,
and $v-\varphi\le 0$ near $x_0$, then either 
$\lambda(A^\varphi(x_0))  \in \RR^n\setminus \overline \Gamma$
or
$
f(\lambda(A^\varphi(x_0)))\le 1
$).  We say that $v$ is a viscosity solution if it is both
a viscosity supersolution and a viscosity subsolution.
\end{Def}

\begin{Def}\label{Def:DegVisSol} A positive continuous function
 $v$ in an open set $\Omega\subset \RR^n$
is a viscosity supersolution (respectively, subsolution) of 
$$
\lambda(A^v) \in \partial \Gamma, \ \ \mbox{in}\ \Omega,
$$
when the following holds: if $x_0\in \Omega$,
$\varphi\in C^2(\Omega)$, $(v-\varphi)(x_0)=0$,
and $v-\varphi\ge 0$ near $x_0$, then
$$
\lambda(A^\varphi(x_0) \in \bar \Gamma
$$
(respectively, if $(v-\varphi)(x_0)=0$,
and $v-\varphi\le 0$ near $x_0$, then either 
$\lambda(A^\varphi(x_0))  \in \RR^n\setminus \overline \Gamma$).  We say that $v$ is a viscosity solution if it is both
a viscosity supersolution and a viscosity subsolution.
\end{Def}

It is clear that for $C^2$ functions the notion of viscosity solutions and classical solutions coincide. Also, viscosity super- and sub-solutions are stable under uniform convergence, see Appendix \ref{Sec:AppVS}.

Note that for any $\lambda = (\lambda_1, \cdots, \lambda_n) \in \Gamma$, $t_1 = \max\lambda_i + 1 > 0$ and $(t_1, \cdots, t_1) \in \lambda + \Gamma_n$. In other words, the sets $\lambda + \Gamma_n$ have non-empty intersection with the ray $\{(t, \cdots, t): t > 0\}$. Thus, if $f^{-1}(1) \neq \emptyset$, then, in view of \eqref{04}, there exists some $c>0$ such that
$f(c,\cdots, c)=1$.  In such situation,  working with 
$\tilde f(\lambda):=f(\frac c2 \lambda)$ instead of $f$, 
we
 may assume without loss of generality the following normalization
condition
\begin{equation}
f(2,\cdots, 2)=1.
\label{F1}
\end{equation}

Let
$$
U(x):= \left(\frac 1{1+|x|^2}\right)^{ \frac {n-2}2 },\qquad x\in \RR^n.
$$
A calculation gives
$$
A^U\equiv 2I.
$$

With the normalization \eqref{F1}, $U$ satisfies
$$
f(\lambda(A^U))=1 \quad\mbox{on}\ \RR^n.
$$ 

For $\bar x\in \RR^n$ and $\mu>0$, let
 $$U^{\bar x, \mu}(x)= \mu U(\mu^{\frac 2{n-2}}(x-\bar x)) = \Big(\frac{\mu^{\frac 2{n-2}}}{1 + \mu^{\frac 4{n-2}} |x- \bar x|^2}\Big)^{\frac{n-2}{2}}.
$$
Note that, in the sense of \eqref{Eq:uvarphi}, $U^{\bar x, \mu} = U_\varphi$ with $\varphi(x) = \mu^{\frac{2}{n-2}}\,(x - \bar x)$. Hence, by the conformal invariance \eqref{Eq:CIProp}, for any $\bar x\in \RR^n$ and $\mu>0$, 
\[
f(\lambda(A^{U_{\bar x, \mu}})) = 1 \text{ on } \RR^n.
\]

\begin{thm}\label{theorem4X}
 Let $(f, \Gamma)$ satisfy \eqref{01weak}-\eqref{0304weak}, \eqref{Eq:ActaSIdiag}-\eqref{02weakY}, \eqref{FU-1a}
 and the normalization condition \eqref{F1}. Let $\epsilon \in (0,1/2]$. There exist constants $\bar m = \bar m(f,\Gamma) \geq 1, \Cfive = \Cfive(f,\Gamma) > 1$, $\daspr = \daspr(\epsilon, f,\Gamma) > 0, \Caspr = \Caspr(\epsilon, f,\Gamma) > 1$  such that for any positive $u \in C^2(B_3(0))$ satisfying\footnote{Note that, $\bar m$ and $\Cfive$ are independent of $\epsilon$.}
 \[
f(\lambda(A^u)) = 1 \text{ in } B_{3}(0) \text{ and } \sup_{B_1(0)} u \geq \Caspr,
\]
there exists $\{x^1, \cdots,  x^m\} \subset B_{2}(0)$ with $1 \leq m \leq \bar m$ satisfying
\begin{enumerate}[(i)]
\item $u(x^1) \geq \sup_{B_1(0)} u$,
\item $|x^i-x^j|\ge \frac{1}{\Cfive}$ for all $1 \leq i\ne j \leq m$,
\item $\frac 1{\Cfive} \le
\frac {   u(x^i)  }{  u(x^j) }\le \Cfive$
 for all $1 \leq i, j \leq m$,
 \item $|u(x) - U^{x^i, u(x^i)}(x)| \leq \epsilon U^{x^i, u(x^i)}(x)$ for all $1 \leq i \leq m$, $x\in B_{\daspr}(x^i)$,
\item $\frac {1}{\Cfive \delta_*^{n-2} u(x^1)}\le u(x) \le \frac {\Cfive} { \delta_*^{n-2}u(x^1)}$ for all $x \in B_{\frac{3}{2}}(0)\setminus \cup_{ i=1}^m
B_{\daspr }(x^i)$,
\item $u(x^i) = \sup_{B_{\daspr}(x^i)} u$.
\end{enumerate}
\end{thm}

\begin{rem}\label{Rem:Dist2B1}
If it holds further that $\sup_{B_1(0)} u > \frac{\Cfive^{1/2}}{\daspr^{\frac{n-2}{2}}}$, then
\[
\min_{1 \leq i \leq m} dist(x^i, B_1(0)) \leq \Big[\sup_{B_1(0)} u\Big]^{-\frac{2}{n-2}}.
\]
To see this, let $x_*$ be a point in $\bar B_1(0)$ such that $u(x_*) = \sup_{B_1(0)} u$. In view of (v) and the stated condition on $\sup_{B_1(0)} u$, $x_*$ belongs to some ball $B_{\daspr}(x^{i_0})$. By (iv), we then have
\[
u(x_*) \leq 2U^{x^{i_0},u(x^{i_0})}(x_*) \leq \frac{1}{|x_* - x^{i_0}|^{\frac{n-2}{2}}},
\]
which implies the assertion.
\end{rem}

Theorem \ref{theorem4X} can be stated equivalently as follows.
\begin{thm}\label{theorem4}
 Let $(f, \Gamma)$ satisfy \eqref{01weak}-\eqref{0304weak}, \eqref{Eq:ActaSIdiag}-\eqref{02weakY}, \eqref{FU-1a}
 and the normalization condition \eqref{F1}.
Assume that  $0 < u_k\in C^2(B_3(0))$ satisfy \eqref{eqB2} and \eqref{eqB1}. Let $\epsilon \in (0,1/2]$. Then there exist $\bar m = \bar m(f,\Gamma) \geq 1, \Cfive(f,\Gamma) > 1$ and $\daspr = \daspr(\epsilon,f,\Gamma) > 0$
such that, 
 after passing to a subsequence, still denoted by $u_k$, there exists $\{x_k^1, \cdots,  x_k^m\}\subset B_{2}(0)$ ($1 \leq m \leq \bar m$) satisfying\footnote{The constant $m$ is independent of $k$.}
\begin{enumerate}[(i)]
\item $u_k(x_k^1) \geq \sup_{B_1(0)} u_k$,
\item $|x_k^i-x_k^j|\ge \frac{1}{\Cfive}$ for all $k \geq 1$, $1 \leq i\ne j \leq m$, 
\item $\frac 1{\Cfive} \le
\frac {   u_k(x_k^i)  }{  u_k(x_k^j) }\le \Cfive$
 for all $k \geq 1$, $1 \leq i, j \leq m$,
 \item 
$|u_k(x) - U^{x_k^i, u_k(x_k^i)}(x)|
\le  \epsilon U^{x_k^i, u_k(x_k^i)}(x)$ for all $k \geq 1$, $1 \leq i \leq m$, $x\in B_{\daspr}(x_k^i)$,
\item $\frac {1}{\Cfive \delta_*^{n-2}u_k(x_k^1)}\le u_k(x) \le \frac {\Cfive} { \delta_*^{n-2} u_k(x_k^1)}$ for all $k \geq 1$, $x \in B_{\frac{3}{2}}(0)\setminus \cup_{ i=1}^m
B_{\daspr }(x_k^i)$,
\item $u_k(x_k^i) = \sup_{B_{\daspr}(x_k^i)} u_k$.
\end{enumerate}
\end{thm}

\begin{rem}
By Remark \ref{Rem:Dist2B1}, we have
\[
dist(\{x_k^1, \ldots, x_k^m\},B_1(0)) \rightarrow 0 \text{ as } k \rightarrow \infty.
\]
\end{rem}

When $(f,\Gamma)=(\frac{1}{2n}\sigma_1, \Gamma_1)$, equation \eqref{eqB2} is $-\Delta u_k = n(n-2)\,u_k^{\frac{n+2}{n-2}}$ and Theorem \ref{theorem4} in this case
was proved by  Schoen \cite{SchoenNotes}. 

See Li \cite{Li95-JDE} and Chen and Lin \cite{ChenLin} for analogous results for the equation $-\Delta u_k = K(x)u_k^{\frac{n+2}{n-2}}$.

In Theorems \ref{theorem4X} and \ref{theorem4}, $B_1(0)$, $B_2(0)$ and $B_3(0)$ can be replaced respectively by $B_{r_1}(0)$, $B_{r_2}(0)$ and $B_{r_3}(0)$, $0 < r_1 < r_2 < r_3$, and in this case the constants $\bar m$, $\Cfive$, $\daspr$ and $\Caspr$ depend also on $r_1$, $r_2$ and $r_3$.

The following is a quantitative version of Theorem \ref{TheoremA}, and is related to Theorem \ref{proposition1} and Theorem \ref{theorem4X}.

\begin{thm}[Quantitative Liouville Theorem]
Let $(f, \Gamma)$ satisfy \eqref{01weak}-\eqref{0304weak}, \eqref{Eq:ActaSIdiag}-\eqref{02weakY}, \eqref{FU-1a} and the normalization
condition \eqref{F1},
and let $\gamma, r_1 > 0$ be constants.
  Then, for every $\epsilon \in (0,1/2]$,
there exist some constants $\daspr > 0, R^* > 0$,  depending only on
$(f, \Gamma)$, $\gamma, r_1$ and $\epsilon$,
 such that if $0< v\in C^2(B_R(0))$ for some $R\ge R^*$,
\begin{equation}
f(\lambda(A^v)=1\quad 
 \mbox{in}\ B_R(0),
\label{AppB1}
\end{equation}
and
\begin{equation}
v\ge \gamma \quad \mbox{in}\ \ B_{r_1}(0),
\label{AppB2}
\end{equation}
then, for some $\bar x\in \RR^n$ satisfying
\begin{equation}
v(\bar x) = \max_{B_{\daspr R}(0)} v \leq \frac{2^{n-1}}{\gamma\,r_1^{n-2}}, \qquad |\bar x| \leq 2^{\frac{1}{n-2}} \gamma^{-\frac{2}{n-2}},
\label{AppB3}
\end{equation}
there holds
\begin{equation}
|v(y)- U^{\bar x, v(\bar x)}(y)|
\le \epsilon\,U^{\bar x, v(\bar x)}(y),\qquad \forall\ |y - \bar x|\le \daspr R.
\label{AppB4}
\end{equation}
\label{quantitativeliouville}
\end{thm}

\begin{rem}
The constant $\daspr$ in Theorems \ref{theorem4X} and \ref{quantitativeliouville} can be chosen the same.
\end{rem}

\begin{rem}
An analogous result for the degenerate elliptic equation $\lambda(A^v) \in \partial \Gamma$ is a consequence of the local gradient estimate \cite[Theorem 1.5]{LiNgBocher}.
\end{rem}

An ingredient in our proof of Theorems \ref{theorem4X} and \ref{theorem4} is the following local gradient estimate, which follows from Theorem \ref{proposition1-deg} and the proof of \cite[Theorem 1.10]{Li09-CPAM}.

\begin{thm}\label{TheoremB}
  Let $(f, \Gamma)$ satisfy \eqref{01weak}-\eqref{0304weak} 
 and let $v\in C^2(B_2(0))$
satisfy, for some constant $b>0$,
\begin{equation}
0<v\le b \qquad \mbox{in}\ B_2(0),
\label{36new}
\end{equation}
and
\begin{equation}
f(\lambda(A^v))=1  \qquad \mbox{in}\ B_2(0).
\label{35new}
\end{equation}
Then, for some constant $C$ depending only on $(f, \Gamma)$  and $b$,
\begin{equation}
|\nabla \ln v|\le C\quad \mbox{in}\ B_1(0).
\label{abc}
\end{equation}
\end{thm}

For $(f,\Gamma)=(\sigma_k^{1/k}, \Gamma_k)$, the result was proved by Guan and Wang \cite{GW03-IMRN}.

When $(f,\Gamma)$ satisfies \eqref{01}-\eqref{04} and is homogeneous of positive degree, Theorem \ref{TheoremB} was proved in \cite{Li09-CPAM}. 

The rest of the paper is organized as follows. We start in Section \ref{sec:Tp1} with the proof of Theorems \ref{proposition1} and \ref{proposition1-deg}. We then prove Theorem \ref{TheoremB} in Section \ref{sec:ThmB}. In Section \ref{sec:T4X}, we first establish an intermediate quantitative Liouville result and then use it to prove Theorem \ref{theorem4X}. In Section \ref{sec:tQL}, we prove Theorem \ref{quantitativeliouville} as an application of Theorem \ref{theorem4X}. In Appendix \ref{Sec:AppA}, we present a lemma about super-harmonic functions which is used in the body of the paper. In Appendix \ref{Sec:AppVS}, we include a relevant remark on the limit of viscosity solutions of elliptic PDE. Finally we collect in Appendix \ref{Sec:AppC} some relevant calculus lemmas.

\section{Non-quantitative Liouville theorems}\label{sec:Tp1}

In this section, we prove Theorems \ref{proposition1} and \ref{proposition1-deg}. We use the method of moving spheres and establish along the way, as a tool, a gradient estimate which is in a sense weaker than that in Theorem \ref{TheoremB} but suffices for the moment. (Note that the proof of Theorem \ref{TheoremB} relies on Theorem \ref{proposition1-deg}.)

\subsection{A gradient estimate}

\begin{thm}\label{SoftGEst}
Let $(f,\Gamma)$ satisfy \eqref{01weak}, \eqref{0304weak} and 
\[
\Gamma \cap \{\lambda + t\mu:  t > 0\} \text{ is convex for all $\lambda \in \Gamma, \mu \in \Gamma_n$}.
\]
Let $0 < v\in C^2(B_2(0))$
satisfy, for some constant $\theta > 1$,
\begin{equation}
\sup_{B_2(0)} v \leq \theta \inf_{B_2(0)} v,
\label{Eq:osclnv}
\end{equation}
and
\[
f(\lambda(A^v))=1  \qquad \mbox{in}\ B_2(0).
\]
Then, for some constant $C$ depending only on $n$  and $\theta$,
\[
|\nabla \ln v|\le C\quad \mbox{in}\ B_1(0).
\]
\end{thm}

This type of gradient estimate was established and used in various work of the first named author and his collaborators under less general hypothesis on $(f,\Gamma)$. It turns out that the same proof works in the current situation. We give a detailed sketch here for completeness.

We use the method of moving spheres as in \cite{LiLi03, LiLi05, LiZhang03, LiZhu95-Duke}. For a function $w$ defined on a subset of $\RR^n$, we define
\[
w_{x,\lambda}(y) = \frac{\lambda^{n-2}}{|y - x|^{n-2}}w\Big(x + \frac{\lambda^2(y -x )}{|y -x|^2}\Big)
\]
wherever the expression makes sense. We will use $w_\lambda$ to denote $w_{0,\lambda}$. We start with a simple result.

\begin{lem} \label{lem-1}
Let $R > 0$ and $w$ be a positive Lipschitz function in $\bar B_R(0)$ such that, for some $L > 0$,
\[
|\ln w(y) - \ln w(z)| \leq L|y - z| \text{ for all } y, z \in \bar B_R(0). 
\]
Then for $\underline{\lambda} = \min(\frac{n-2}{2L},\frac{R}{2})$ we have 
\begin{equation}
w_\lambda \le  w\
\mbox{in}\ B_{\underline{\lambda}}(0)\setminus B_\lambda, 
\forall 0<\lambda<
\underline{\lambda}.
\label{aa-1}
\end{equation}
\end{lem}

\begin{proof}
Write  $w$
in polar coordinates $w(r, \theta)$.
It is easy to see that \eqref{aa-1}
is equivalent to
\begin{equation}
r^{\frac {n-2}2} w(r, \theta))\le \ 
s^{\frac {n-2}2} w(s, \theta),
\ \ \forall\ 0<r<s<\underline{\lambda},\ \forall\ \theta.
\label{aa-2}
\end{equation}
Estimate \eqref{aa-2} is readily seen from the estimates
\begin{align*}
\ln w(s, \theta) -\ln  w(r, \theta)
	&\geq  - L|s - r|
		\geq - \frac{n-2}{2\underline{\lambda}} (s - r)\\
	&\geq - \frac{n-2}{2}[\ln s - \ln r].
\end{align*}
Lemma \ref{lem-1} is established.
\end{proof}

\begin{proof}[Proof of Theorem \ref{SoftGEst}]
By Lemma \ref{lem-1}, there exists some $r_0 \in (0,1/3)$ such that
\[
v_{x,\lambda} \leq v \text{ in } B_{r_0}(x) \setminus B_{\lambda}(x) \text{ for all } \lambda \in (0,r_0) \text{ and } x \in B_{4/3}(0).
\]
It is easy to see that, for some $r_1 \in (0,r_0)$, 
\[
v_{x,\lambda} \leq v \text{ in } B_{5/3}(0) \setminus B_{r_0}(x) \text{ for all } \lambda \in (0,r_1) \text{ and } x \in B_{4/3}(0).
\]
We then define, for $x \in B_{4/3}(0)$,
\[
\bar\lambda(x) = \sup\Big\{\lambda \in (0, 5/3 - |x|): v_{x,\lambda} \leq v \text{ in } B_{5/3}(0) \setminus B_{\lambda}(x)\Big\}.
\]

We have
\[
v_{x,\bar \lambda(x)} \le v\
\mbox{in}\ B_{5/3}(0)\setminus B_{\bar \lambda(x)}(x), 
\]
By the conformal invariance \eqref{Eq:CIProp}, 
$v_{x,\bar \lambda(x)}$ satisfies
\[
f(\lambda(A^{v_{x,\bar \lambda(x)}}))=1,
\quad \ \ \mbox{in}\ B_{5/3}(0)\setminus \overline{  B_{\bar \lambda(x)}(x) }.
\]
Using the above two displayed equations, the definition
of $\bar \lambda(x)$,
and using the ellipticity of the
equation satisfied by $v$ and $v_{x,\bar \lambda(x)}$,
we can apply 
the strong maximum principle and Hopf Lemma to infer that either $\bar\lambda(x) = 5/3 - |x|$ or there exists 
some $y \in \partial B_{5/3}(0)$ such that
$$
v_{x,\bar \lambda(x)}(y)=v(y)
$$
--- see the proof of \cite[Lemma 4.5]{LiLi05}.

In the latter case, \eqref{Eq:osclnv} implies that
\[
\theta \geq \frac{v\big(x + \frac{\bar\lambda(x)^2(y-x)}{|y - x|^2}\big)}{v(y)} = \frac{|y - x|^{n-2}}{\bar\lambda(x)^{n-2}} \geq \frac{(5/3 - |x|)^{n-2}}{\bar\lambda(x)^{n-2}}.
\]
In either case, we obtain that
\[
\bar \lambda(x) \geq c(n,\theta) > 0 \text{ for all } x \in B_{4/3}(0).
\]
The conclusion then follows from \cite[Lemma A.2]{LiLi05}.
\end{proof}

\subsection{Proof of Theorem \ref{proposition1}}

\begin{rem} \label{Rem:ASD}
If we drop condition \eqref{Eq:ActaSIdiag}, the case that $v$ is constant in Theorem \ref{proposition1} can occur. For example, consider $n \geq 3$ and 
\begin{align*}
f(\lambda) &= \sigma_2(\lambda) + 1, \qquad \Gamma = \Big\{\lambda \in \RR^n\Big| \sum_{j=1}^n \lambda_j - \lambda_i > 0  \text{ for all } i = 1, \ldots, n\Big\},\\
v_k(x) &= \Big(\frac{1}{R_k^{n+4}}|x - x_k|^{-\frac{n-4}{2}} + 1\Big)^{\frac{2(n-2)}{n-4}} \text{ for } |x| < \frac{|x_k|}{2} = R_k \rightarrow \infty.
\end{align*}
It is readily seen that $f(t, \ldots, t) > 1$ for all $(t, \ldots, t) \in \Gamma$, $v_k$ satisfies \eqref{ve} (cf. \cite[Theorem 1.6]{LiNgBocher}) and $v_k \rightarrow 1$ in $C^{0}_{loc}(\RR^n)$.
\end{rem}

\medskip
\begin{proof}[Proof of Theorem \ref{proposition1}]
We may assume that $R_k\ge 5$ for all $k$.

Clearly, for every $\beta>1$, there exists some positive constant 
$C(\beta)$, independent of $k$, such that 
$1/C(\beta)\le 
v_k\le C(\beta)$ in $B_\beta(0)$. 
 It follows from Theorem \ref{SoftGEst}
that $|\nabla \ln v_k|
\le C'(\beta)$ in $B_{\beta/2}(0)$.
It follows, after passing to a subsequence, that for every $0<\alpha<1$,
$v_k\to v$ in $C^{\alpha}_{loc}(\RR^n)$, $v\in C^{0,1}_{loc}(\RR^n)$ and $v$ is super-harmonic on $\RR^n$.

Using the positivity, the superharmonic of $v$, and the
maximum principle, we can find $c_0 > 0$ such that  
\begin{equation}
v(y)\ge
2c_0 |y|^{2-n} ,\ \ \forall\ |y|\ge 1.
	\label{Eq:vSuperhar}
\end{equation}
Passing to a subsequence and shrinking $R_k$ and $c_0 > 0$, if necessary, we may assume that
\begin{equation}
|v_k(y)-v(y)|\le R_k^{-n},\quad
\ \ \forall\ |y|\le R_k,
\label{aa-0}
\end{equation}
and
\begin{equation}
v_k(y) \geq c_0(1+|y|)^{2-n}, \ \ \forall\   |y| < R_k.
\label{aa-0Ex}
\end{equation}

\begin{lem} \label{lem-1new}
Under the hypotheses of Theorem \ref{proposition1},
there 
 exists a function $\lambda^{(0)}: 
\RR^n \rightarrow (0,\infty)$ such that, for all $k$,
$$
(v_k)_{x,\lambda} \le v_k\
\mbox{in}\ B_{R_k}(0)\setminus B_\lambda(x), \forall \ 0<\lambda<
\lambda^{(0)}(x)\
\mbox{and}\  |x|\le \frac {R_k}5.
$$
\end{lem}

\begin{proof}
For $|x| \leq \frac{R_k}{5}$, we have, by \eqref{aa-0} and \eqref{aa-0Ex}, for all $k$ that
\begin{equation}
\frac{1}{c_1(x)}\le  v_k\le c_1(x)\ \mbox{in}\ B_{4r_1(x)}(x) \subset B_{R_k}(0), \label{bound1}
\end{equation}
where 
\begin{align*}
r_1(x) = \frac{1}{4} + |x|, \text{ and }
c_1(x) = \max \Big\{1+ \sup_{B_{4r_1(x)}(x)} v, \frac{1}{c_0} (1 + |x| + 4r_1(x))^{n-2}\Big\}.
\end{align*}

By \eqref{bound1} and Theorem \ref{SoftGEst},
there exists $c_2(x)>0$, independent of $k$,  such that
$$
|\nabla \ln v_k|\le c_2(x)\ \mbox{in}\ B_{2r_1(x)}(x).
$$
Thus, by Lemma \ref{lem-1}, we can find $0 <  \lambda_1(x) 
< r_1(x)$ independent of $k$ such that
\begin{equation}
(v_k)_{x,\lambda} \le v_k\
\mbox{in}\ B_{\lambda_1(x)}(x)\setminus B_\lambda(x) \ \ \forall\ 
 0 < \lambda < \lambda_1(x).
\label{bound2}
\end{equation}

For $0 < \lambda < \lambda_1(x)$,  we have, using \eqref{bound1}, that
\begin{equation}
(v_k)_{x,\lambda}(y) \leq 
 \frac{\lambda^{n-2}c_1(x)}{|y - x|^{n-2}}  ,
\quad \forall \ y \in  B_{R_k}(0)\setminus B_{\lambda}(x).
	\label{Eq:B2-1}
\end{equation}
When $y \in B_{R_k}(0) \setminus B_{4r_1(x)}(x)$, $\frac{1}{2}(1 + |y|) < |y - x|$ and we obtain, using \eqref{Eq:B2-1} and \eqref{aa-0Ex}, that
\begin{equation}
(v_k)_{x,\lambda}(y) \leq \frac{(2\lambda)^{n-2}c_1(x)}{(1 + |y|)^{n-2}} \leq \frac{(2\lambda)^{n-2}c_1(x)}{c_0}\,v_k(y).
	\label{Eq:B3-1}
\end{equation}
When $y \in B_{4r_1(x)}(x) \setminus B_{\lambda_1(x)}(x)$, $1 + |y| \leq 2(1 + 3|x|)$, $|y - x| \geq \lambda_1(x)$ and we obtain, using \eqref{Eq:B2-1} and \eqref{aa-0Ex}, that
\begin{equation}
(v_k)_{x,\lambda}(y) \leq \frac{\lambda^{n-2} c_1(x)}{\lambda_1(x)^{n-2}} \leq \frac{(2\lambda)^{n-2} c_1(x) (1 + 3|x|)^{n-2}}{\lambda_1(x)^{n-2} c_0} v_k(y).
	\label{Eq:B4-1}
\end{equation}
Letting
\[
\lambda^{(0)}(x) = \min \Big\{ \lambda_1(x), \frac{\lambda_1(x)}{2(1 + 3|x|)}\Big[\frac{c_1(x)}{c_0}\Big]^{\frac{1}{n-2}}\Big\} \leq \lambda_1(x),
\]
we derive from \eqref{Eq:B3-1} and \eqref{Eq:B4-1} that
\begin{equation}
(v_k)_{x,\lambda} \leq 
 v_k \ \ \ \mbox{in}\
 B_{R_k}(0) \setminus B_{\lambda_1(x)}(x) 
\text{ and } 0 < \lambda < \lambda^{(0)}(x).
\label{bound3}
\end{equation}
Lemma \ref{lem-1new} follows from
\eqref{bound2} and \eqref{bound3}.
\end{proof}

\bigskip

Define, 
 for $x\in \RR^n$ and $|x|\le  R_k/5$,
that 
\[
\bar \lambda_k(x)
=\sup\Big\{0<\mu\le \frac {R_k}5\  |\
(v_k)_{x,\lambda} \le v_k\
\mbox{in}\ B_{R_k}(0)\setminus B_\lambda(x), \forall 0<\lambda<\mu\Big\}.
\]

\bigskip

By Lemma \ref{lem-1new},
$$\bar\lambda(x):=
\displaystyle{
\liminf_{k\to\infty}\bar\lambda_k(x)
}
\in [\lambda^{(0)}(x), \infty], \qquad x\in \RR^n.
$$

By \eqref{Eq:vSuperhar},
 $$
\alpha:= \liminf_{|y|\to \infty}
|y|^{n-2}v(y)\in (0, \infty].
$$

\begin{lem}\  Assume \eqref{01weak}-\eqref{0304weak}. Then either $v$ is constant or
$$
\alpha=\lim_{|y|\to \infty}
|y|^{n-2}v_{x,\bar\lambda(x)}(y)=
\bar \lambda(x)^{n-2}v(x)<\infty,\qquad
\forall\ x\in \RR^n.
$$
\label{lem-0.3}
\end{lem}

\begin{proof}

\noindent{\bf Step 1.}\ If $\bar \lambda(x)<\infty$ for some $x\in \RR^n$,
then 
$$
\alpha= \lim_{|y|\to \infty}
|y|^{n-2}v_{x,\bar\lambda(x)}(y)
=\bar \lambda(x)^{n-2}v(x)
<\infty.
$$

Since
 $\bar \lambda(x)<\infty$,
we have,   along a subsequence, $\bar \lambda_k(x)\to \bar \lambda(x)$ ---
but for simplicity, we still use $\{\bar \lambda_k(x)\}$,
 $\{v_k\}$, etc to denote the
subsequence.
By the definition of $\bar\lambda_k(x)$, we have
\begin{equation}
(v_k)_{x,\bar \lambda_k(x)} \le v_k\
\mbox{in}\ B_{R_k}(0)\setminus B_{\bar \lambda_k(x)}(x), 
\label{vkeq}
\end{equation}
By the conformal invariance \eqref{Eq:CIProp}, 
$(v_k)_{x,\bar \lambda_k(x)}$ satisfies
\begin{equation}
f(\lambda(A^{(v_k)_{x,\bar \lambda_k(x)}}))=1,
\quad \ \ \mbox{in}\ B_{R_k}(0)\setminus \overline{  B_{\bar \lambda_k(x)}(x) }.
\label{vkeqnew}
\end{equation}
Using  \eqref{ve}, \eqref{vkeq}, \eqref{vkeqnew}, the definition
of $\bar \lambda_k(x)$,
and using the ellipticity of the
equation satisfied by $v_k$ and $(v_k)_{x,\bar \lambda_k(x)}$,
we can apply 
the strong maximum principle and Hopf Lemma to infer the existence of 
some $y_k\in \partial B_{R_k}(0)$ such that
$$
(v_k)_{x,\bar \lambda_k(x)}(y_k)=v_k(y_k)
$$
--- see the proof of \cite[Lemma 4.5]{LiLi05}.

It follows 
that
$$
\lim_{k\to\infty} |y_k|^{n-2} v_k(y_k)
= \lim_{k\to\infty} |y_k|^{n-2} (v_k)_{x,\bar \lambda_k(x)}(y_k)
=(\bar \lambda(x))^{n-2}v(x).
$$
This implies, in view of \eqref{aa-0}, that
$$
\alpha\le \lim_{k\to\infty} |y_k|^{n-2} v(y_k)=\bar \lambda(x)^{n-2}v(x)
=\lim_{|y|\to \infty}
|y|^{n-2}v_{x,\bar\lambda(x)}(y)<\infty.
$$

On the other hand, if $\hat y_i$ is such that
$|\hat y_i|\to \infty$ and
$$
\alpha= \lim_{i\to\infty} |y_i|^{n-2}v(y_i),
$$
then, since 
$v_{ x, \bar \lambda(x) }\le v$
in $\RR^n\setminus B_{\bar \lambda(x) }(x)$, we have
$$
v(y_i)\ge \frac {   \bar \lambda(x) ^{n-2} }
{  |y_i-x|^{n-2}  }
v\left(x+\frac {\bar\lambda(x)^2(y_i-x) }{   |y_i-x|^2  }\right).
$$
This gives
$$
\alpha= \lim_{i\to\infty} |y_i|^{n-2}v(y_i)
\ge \bar \lambda(x)^{n-2}v(x).
$$

Step 1 is established.

\medskip

\noindent{\bf Step 2.}\  It remains to show that either $v$ is constant or, for every $x\in \RR^n$,
$\bar \lambda(x)<\infty$.

\medskip

To this end, we show that if $\bar \lambda(x)=\infty$ for some $x \in \RR^n$, then $v$ is constant. Indeed, assume that  $\bar \lambda_k(x)
\to \infty$ as $k\to\infty$.
We easily derive from this and the convergence of $v_k$ to $v$ that
\begin{equation}
v_{x,\lambda}\le v\quad \mbox{in}\
\RR^n\setminus B_\lambda(x)\ \ \forall\ \lambda>0.
\label{vx}
\end{equation}
The above is equivalent to the property that for every
fixed unit vector $e$, 
$r^{ \frac{n-2}2 }v(x+ re)$ is non-decreasing
in $r$.  Thus
$$
r^{n-2} \min_{ \partial B_r(x) }v\ge 
r^{ \frac {n-2}2 }
\min_{ \partial B_1(x) }v\quad \forall\ r\ge 1.
$$
In particular,
$\alpha=\liminf_{ |y|\to \infty} |y|^{n-2}v(y)=\infty$.
This implies, by Step 1, that 
$\bar\lambda(x)=\infty$ for {\it every} $x\in \RR^n$, and therefore
\eqref{vx} holds for {\it every} $x\in \RR^n$.
This implies that
$v$ is a constant, see Corollary \ref{lem-app1}.
\end{proof}

\begin{lem}
Assume \eqref{01weak}-\eqref{0304weak} and \eqref{Eq:ActaSIdiag}. Then the function $v$ in Theorem \ref{proposition1} cannot be constant.
\end{lem}

\begin{proof}
Fix some $t > 0$ for the moment. Set $\varphi(x) = v(0) - t\,|x|^2$ and fix some $r > 0$ such that $\varphi > 0$ in $B_{r}(0)$ and $\varphi < v_k$ on $\partial B_r(0)$ for all sufficiently large $k$. Let
\[
\gamma_k = \sup_{B_r(0)} (\varphi - v_k) \text{ and } \varphi_k = \varphi - \gamma_k.
\]
Then $\varphi_k \leq v_k$ in $B_r(0)$ and $\varphi_k(x_k) = v_k(x_k)$ for some $x_k \in \bar B_r(0)$. Noting that
\[
\gamma_k = (\varphi - v)(x_k) - (v_k - v)(x_k)    = -t|x_k|^2 - (v_k - v)(x_k) 
\]
and
\[
\lim_{k \rightarrow \infty} \gamma_k = \sup_{B_r(0)} (\varphi - v) = 0
\]
we deduce that $x_k \rightarrow 0$. This leads to
\[
\varphi_k(x_k) = v_k(x_k), \nabla \varphi_k(x_k) = \nabla v_k(x_k), \nabla^2\varphi_k(x_k) \leq \nabla^2 v_k(x_k)
\]
and
\[
A^{\varphi_k}(x_k) \geq A^{v_k}(x_k).
\]
Noting that there is some $C > 0$ independent of $\delta$ and $k$ such that, for large $k$, 
\[
A^{\varphi_k}(x_k) \leq \Big(\frac{4}{n-2}v(0)^{-\frac{n+2}{n-2}}t + C\delta\Big)I.
\]
Thus, we can select $t$ and $\delta$ such that
\[
A^{v_k}(x_k) \leq A^{\varphi_k}(x_k) \leq t_0\,I.
\]
where $t_0$ is the constant in \eqref{Eq:ActaSIdiag}. Since $f(\lambda(A^{v_k}(x_k))) = 1$, this contradicts \eqref{02weakX}, \eqref{0304weak} and \eqref{Eq:ActaSIdiag}.
\end{proof}

Recall that
$0<v\in C^{0,1}_{loc}(\mathbb R^n)$,
$\Delta v\le 0\ \mbox{in}\ \mathbb R^n$,
and it remains to consider the case that, for every $x\in\mathbb R^n$, there exists $0<\bar\lambda(x)<\infty$ such that
\[
v_{x,\bar\lambda(x)}(y)\le v(y),\ \forall\ |y-x|\ge \bar\lambda(x),
\]
and
\[
\lim_{|y|\to \infty}
|y|^{n-2}v_{x,\bar\lambda(x)}(y)=
\alpha:=\liminf_{|y|\to\infty}|y|^{n-2}v(y)<\infty.
\]

If $v$ is in $C^2(\mathbb R^n)$, the conclusion of Theorem \ref{proposition1}
follows from the proof of theorem 1.3 in \cite{LiLi05}.
An observation made in \cite{Li07-ARMA} easily
allows the proof to hold for $v\in C^{0,1}_{loc}(\mathbb R^n)$.
For readers' convenience, we outline the proof below.

Let
$\psi(y)=\frac y{|y|^2}$. We denote
$$
v_\psi := |J_\psi|^{ \frac {n-2}{2n}} v\circ \psi.
$$
We know that
$\lambda(A^{ v_\psi }(y))=\lambda(A^v (\psi (y))$. Namely, $A^{ v_\psi }(y)$ and $ A^v (\psi (y))$ differ only
by an orthogonal conjugation.

Introduce
\[
w^{(x)}:=(v_{x,\bar\lambda(x)})_\psi,\ x\in \RR^n.
\]
We deduce from the above properties of $v$ that
for every $x\in \mathbb R^n$, there exists some $\delta(x)>0$ such that
\[
v_\psi\ge w^{(x)}\ \mbox{in}\ B_{\delta(x)}(0)\setminus\{0\},
\]
\[
w^{(x)}(0)=\alpha=\liminf_{y\to 0}v_\psi(y),
\]
\[
\Delta v_\psi\le 0,\ \mbox{in}\ B_{\delta(x)}(0)\setminus\{0\},
\]

 Let
$D=\{x\in \RR^n\ |\ v\ \mbox{is differentiable at}
 \ x\}$.  Since $v\in C^{0,1}_{loc}(\RR^n)$, the  Lebesgue measure
of $\mathbb R^n\setminus D$ is $0$.
It is clear that 
 $w^{(x)}(y)$ is differentiable at $y=0$
if $v$ is differentiable at $x$.

By \cite[Lemma 4.1]{LiLi05}, 
\[
\nabla w^{(x)}(0)=\nabla w^{(\tilde x)}(0),\ \ \forall\ x, \tilde x\in D.
\]
Namely, for some $V\in \RR^n$,
\[
\nabla w^{(x)}(0)=V,\ \forall\ x\in D.
\]

A calculation yields
\[
\nabla w^{(x)}(0)=(n-2)\alpha x+\alpha^{ \frac n{n-2} }
v(x)^{ \frac n{n-2} }
\nabla v(x).
\]
 Thus
\[
\nabla_x\big( \frac {n-2}2 \alpha^{ \frac n{n-2} }
v(x) ^{ -\frac 2{n-2} }
- \frac {n-2}2 \alpha |x|^2 +V\cdot x \big)=0,\ \forall\ x\in D.
\]

 Consequently, for some $\bar x\in \RR^n$ and $d\in \RR$,
\[
v(x) ^{ -\frac 2{n-2} }
\equiv
\alpha^{ -\frac 2{n-2} }
|x-\bar x|^2+d \alpha^{ -\frac 2{n-2} }.
\]

 Since $v>0$, we must have $d>0$, so
\[
v(x)\equiv \big( \frac {  \alpha^{ \frac 2{n-2} }  }
{d+|x-\bar x|^2}\big) ^{ \frac {n-2}2 }.
\]
We have proved that $v$ is of the form
\eqref{vform}
for some $\bar x\in \RR^n$ and some positive constants
$a$ and $b$. 

To finish the proof, we show that $f(2b^2a^{-2}, \ldots, 2b^2a^{-2}) = 1$ when \eqref{Eq:ActaSIdiag} and \eqref{02weakY} are in effect. For $\delta > 0$, let
\[
v_\delta(x) = v(x) - \delta|x|^2.
\]
Since $v_k \rightarrow v$ in $C^0(\bar B_\delta(0))$, there exists $\beta_k \rightarrow 0$ and $x_k \rightarrow 0$ such that $\hat v_k := v_\delta + \beta_k$ satisfies
\[
(v_k - \hat v_k)(x_k) = 0 \text{ and } v_k - \hat v_k \geq 0 \text{ near } x_k.
\]
We have $A^{\hat v_k}(x_k) \geq A^{v_k}(x_k)$. Therefore, by \eqref{02weakY}, $\lambda(A^{\hat v_k}(x_k)) \in \Gamma$, and by \eqref{0304weak}, 
\begin{equation}
f(\lambda(A^{\hat v_k}(x_k))) \geq f(\lambda(A^{v_k}(x_k))) = 1.
	\label{Eq:13IV2016}
\end{equation}
Noting that $A^{\hat v_k}(x_k) \rightarrow 2b^2a^{-2}I$ as $\delta \rightarrow 0, k \rightarrow \infty$, we infer that $2b^2a^{-2} > t_0$. (Indeed, if $2b^2a^{-2} \leq t_0$, then, for small $\rho > 0$, we have $A^{\hat v_k}(x_k) < (t_0 + \rho)I$ for small $\delta$ and large $k$, which implies, by \eqref{Eq:13IV2016}, \eqref{02weakX} and \eqref{0304weak}, that $1 \leq f(\lambda(A^{\hat v_k}(x_k)) < f(t_0 + \rho, \ldots, t_0 + \rho)$, which contradicts \eqref{Eq:ActaSIdiag}.) In view of \eqref{02weakY}, this implies that $(2b^2a^{-2}, \ldots, 2b^2a^{-2}) \in \Gamma$. We can now send $k \rightarrow \infty$ and then $\delta \rightarrow 0$ in \eqref{Eq:13IV2016} to obtain
\[
f(\lambda(A^v(0))) \geq 1, \text{ i.e. } f(2b^2a^{-2}, \ldots, 2b^2a^{-2}) \geq 1.
\]

Using $v^\delta(x) = v(x) + \delta|x|^2$ instead of $v_\delta$ and the fact that $(2b^2a^{-2}, \ldots, 2b^2a^{-2}) \in \Gamma$, one can easily derive
\[
f(2b^2a^{-2}, \ldots, 2b^2a^{-2}) \leq 1.
\]
Theorem \ref{proposition1} is established.
\end{proof}

\subsection{Proof of Theorem \ref{proposition1-deg}}\label{sec:Tp1deg}

\begin{proof}[Proof of Theorem \ref{proposition1-deg}]
We start with some preparation as in the proof of Theorem \ref{proposition1}. We may assume that $R_k \geq 5$ for all $k$.

By hypotheses, $v_k$ is super-harmonic and positive on $\RR^n \setminus \{0\}$. Therefore, $v_*$ is super-harmonic and non-negative on $\RR^n \setminus \{0\}$. Hence either $v_* \equiv 0$ or $v_* > 0$ in $\RR^n \setminus \{0\}$. In the former case we are done. We assume henceforth that the latter holds.

Now, for every $\beta>2$, there exists some positive constant 
$C(\beta)$, independent of $k$, such that 
$C(\beta)^{-1} \leq M_k v_k\le C(\beta)$ in $B_\beta(0) \setminus B_{1/\beta}(0)$. It follows from Theorem \ref{SoftGEst}
that $|\nabla \ln v_k|
\le C'(\beta)$ in $B_{\beta/2}(0) \setminus B_{2/\beta}(0)$.
It follows, after passing to a subsequence, that for every $0<\alpha<1$,
$M_k v_k\to v_*$ in $C^{\alpha}_{loc}(\RR^n \setminus \{0\})$, $v_*\in C^{0,1}_{loc}(\RR^n \setminus \{0\})$ and
\begin{equation}
|\nabla \ln v_*|
\le C'(\beta) \text{ in } B_{\beta/2}(0) \setminus B_{2/\beta}(0).
	\label{Eq:v*Harnack}
\end{equation}

By the super-harmonicity and the positivity of $v_*$, we can find $c_0 > 0$ such that 
\begin{equation}
v_*(y)\ge
2c_0 |y|^{2-n} ,\ \ \forall\ |y|\ge 1.
	\label{Eq:vSuperhardeg}
\end{equation}
Hence, passing to a subsequence and shrinking $R_k$ and $c_0 > 0$ if necessary, we can assume without loss of generality that, for all $k$,
\begin{equation}
|M_kv_k(y)-v_*(y)|\le (R_k)^{-n},\quad
\ \ \forall\ R_k^{-1}\le |y|\le R_k
\label{aadeg-0}
\end{equation}
and
\begin{equation}
M_kv_k(y) \geq c_0(1+|y|)^{2-n} \ \ \forall 0< |y| < R_k.
\label{aadeg-0Ex}
\end{equation}

Denote
$$
(v_k)_{x,\lambda}(y):=(\frac \lambda {|y-x|})^{n-2}v_k(x+
\frac {\lambda^2(y-x)} { |y-x|^2}),
$$
 the Kelvin transformation of $v_k$.
We use $(v_k)_\lambda$ to denote $(v_k)_{0, \lambda}$.

\begin{lem} \label{lem-1deg}
Under the hypotheses of Theorem \ref{proposition1-deg},
there 
 exists a function $\lambda^{(0)}: 
\RR^n \setminus \{0\} \rightarrow (0,\infty)$ such that $\lambda^{(0)}(x) \leq |x|$ and, for all $k$,
$$
(v_k)_{x,\lambda} \le v_k\
\mbox{in}\ B_{R_k}(0)\setminus (B_\lambda(x) \cup \{0\}), \forall \ 0<\lambda<
\lambda^{(0)}(x)\
\mbox{and}\  |x|\le \frac {R_k}5.
$$
\end{lem}

\begin{proof}
We adapt the proof of Lemma \ref{lem-1new}. For $0 < |x| < \frac{R_k}{5}$, we have, by \eqref{aadeg-0} and \eqref{aadeg-0Ex},  for all $k$ that
\begin{equation}
\frac{1}{c_1(x)}\le  M_k v_k\le c_1(x)\ \mbox{in}\ B_{4r_1(x)}(x) \subset B_{R_k}(0),
\label{bound1-deg}
\end{equation}
where
\begin{align*}
r_1(x) = \frac{1}{8}|x| \text{ and } c_1(x) = \max\Big\{1 + \sup_{B_{4r_1(x)}(x)} v_*, \frac{1}{c_0}(1 + 2|x|)^{n-2}\Big\}.
\end{align*}

By Theorem \ref{SoftGEst} and \eqref{bound1-deg}, there exists $c_2(x)>0$, independent of $k$,  such that
$$
|\nabla \ln v_k|\le c_2(x)\ \mbox{in}\ B_{2r_1(x)}(x).
$$
Thus, by Lemma \ref{lem-1}, we can find $0 < \lambda_1(x) < r_1(x)$ independent of $k$ such that
\begin{equation}
(v_k)_{x,\lambda} \leq v_k \text{ in } B_{\lambda_1(x)}(x) \setminus (B_\lambda(x) \cup \{0\}) \text{ for all } 0 < \lambda < \lambda_1(x).
	\label{Eq:E2-1}
\end{equation}
For $0 < \lambda < \lambda_1(x)$, we have, using \eqref{bound1-deg}, that
\begin{equation}
(v_k)_{x,\lambda}(y) \leq \frac{\lambda^{n-2}c_1(x)}{|y - x|^{n-2}} \text{ for } y \in B_{R_k}(0) \setminus (B_{\lambda}(x) \cup\{0\}).
	\label{Eq:E3-1}
\end{equation}
For $y \in B_{R_k}(0) \setminus (B_{1 + 4|x|}(x) \cup \{0\})$, we have $\frac{1}{2}(1 + |y|) \leq |y - x|$ and we obtain, using \eqref{Eq:E3-1} and \eqref{aadeg-0Ex}, that
\begin{equation}
(v_k)_{x,\lambda}(y) \leq \frac{(2\lambda)^{n-2} c_1(x)}{c_0} v_k(y).
	\label{Eq:E3-2}
\end{equation}
For $y \in B_{1 + 4|x|}(x) \setminus  (B_{\lambda_1(x)}(x) \cup \{0\})$, we have $1 + |y| \leq 2(1+3|x|)$, $|y - x| \geq \lambda_1(x)$ and we obtain, using \eqref{Eq:E3-1} and \eqref{aadeg-0Ex}, that
\begin{equation}
(v_k)_{x,\lambda}(y) \leq \frac{(2\lambda)^{n-2} c_1(x) (1 + 3|x|)^{n-2}}{c_0 \lambda_1(x)^{n-2}} v_k(y).
	\label{Eq:E4-1}
\end{equation}

Letting
\[
\lambda^{(0)}(x) = \min \Big\{ \lambda_1(x), \frac{\lambda_1(x)}{2(1+3|x|)} \Big[\frac{c_1(x)}{c_0}\Big]^{\frac{1}{n-2}}\Big\} \leq \lambda_1(x),
\]
we see that the conclusion of Lemma \ref{lem-1deg} follows from \eqref{Eq:E2-1}, \eqref{Eq:E3-2} and \eqref{Eq:E4-1}.
\end{proof}

Define, 
 for $0 < |x|\le  R_k/5$,
that 
\[
\bar \lambda_k(x)
=\sup\Big\{0<\mu\le \min(|x|,\frac {R_k}5)\  |\
(v_k)_{x,\lambda} \le v_k\
\mbox{in}\ B_{R_k}(0)\setminus (B_\lambda(x) \cup \{0\}), \forall 0<\lambda<\mu\Big\}.
\]

\bigskip

By Lemma \ref{lem-1deg},
$$\bar\lambda(x):=
\displaystyle{
\liminf_{k\to\infty}\bar\lambda_k(x)
}
\in [\lambda^{(0)}(x), |x|], \qquad x\in \RR^n \setminus \{0\}.
$$
Clearly,
\[
(v_*)_{x,\bar\lambda(x)} \leq v_* \text{ in } \RR^n \setminus (B_{\bar\lambda(x)}(x) \cup \{0\}) \text{ for all } x \in \RR^n \setminus \{0\}.
\]

We have a dichotomy:
\begin{align}
\text{ either } &\bar \lambda(x) = |x| \text{ for all }x \in \RR^n \setminus \{0\},
	\label{Eq:Casei}\\
\text{ or } &\bar\lambda(x_0) < |x_0| \text{ for some }x_0 \in \RR^n \setminus \{0\}.
	\label{Eq:Caseii}
\end{align}
In case \eqref{Eq:Casei}, we obtain that $v_*$ is radially symmetric about the origin thanks to Lemma \ref{lem-app1Sing}. To finish the proof, we assume in the rest of the argument that \eqref{Eq:Caseii} holds and derive a contradiction.

We first collect some properties of $\bar\lambda(x)$. We start with an analogue of Lemma \ref{lem-0.3}. By \eqref{Eq:vSuperhardeg}, let
 $$
\alpha:= \liminf_{|y|\to \infty}
|y|^{n-2}v_*(y)\in (0, \infty].
$$

\begin{lem}\  Under the hypotheses of Theorem \ref{proposition1-deg}, if $\bar \lambda(x)< |x|$ for some $x\in \RR^n \setminus \{0\}$,
then 
$$
\alpha= \lim_{|y|\to \infty}
|y|^{n-2}(v_*)_{x,\bar\lambda(x)}(y)
=\bar \lambda(x)^{n-2}v_*(x)
<\infty.
$$
\label{lem-0.3deg}
\end{lem}

\begin{proof} We adapt Step 1 in the proof of Lemma \ref{lem-0.3}. Assume that $\bar \lambda(x)<|x|$ and (without loss of generality) that $\bar \lambda_k(x)\to \bar \lambda(x)$. Arguing as before but using the strong maximum principle for solutions with isolated singularities \cite[Theorem 1.6]{Li06-JFA} instead of the standard strong maximum principle, this leads to the existence of 
some $y_k\in \partial B_{R_k}(0)$ such that
$$
(v_k)_{x,\bar \lambda_k(x)}(y_k)=v_k(y_k).
$$

It follows 
that
$$
\lim_{k\to\infty} |y_k|^{n-2} M_k v_k(y_k)
= \lim_{k\to\infty} |y_k|^{n-2} (M_k v_k)_{x,\bar \lambda_k(x)}(y_k)
=(\bar \lambda(x))^{n-2}v_*(x).
$$
This implies, in view of \eqref{aadeg-0}, that
$$
\alpha\le \lim_{k\to\infty} |y_k|^{n-2} v_*(y_k)=\bar \lambda(x)^{n-2}v_*(x)
=\lim_{|y|\to \infty}
|y|^{n-2}(v_*)_{x,\bar\lambda(x)}(y)<\infty.
$$
On the other hand, as in the proof of Lemma \ref{lem-0.3}, we can use $(v_*)_{ x, \bar \lambda(x) }\le v_*$
in $\RR^n\setminus (B_{\bar \lambda(x) }(x) \cup \{0\})$ to show that 
$$
\alpha
\ge \bar \lambda(x)^{n-2}v_*(x).
$$
The conclusion is readily seen.
\end{proof}

\begin{lem}\label{Lem:lamCont}
Under the hypotheses of Theorem \ref{proposition1-deg}, if $\bar\lambda(x_0) < |x_0|$ for some $x_0 \in \RR^n \setminus \{0\}$, then
\[
\limsup_{x \rightarrow x_0} \bar\lambda(x) \leq \bar\lambda(x_0).
\]
\end{lem}

\begin{proof} Along a subsequence, we have $\bar\lambda_k(x_0) \rightarrow \bar\lambda(x_0)$.

As in the proof of Lemma \ref{lem-0.3deg}, there exists $y_k\in \partial B_{R_k}(0)$ such that
\begin{equation}
(v_k)_{x_0,\bar \lambda_k(x_0)}(y_k)=v_k(y_k).
	\label{Eq:F2-1}
\end{equation}

We know that
\[
\eta_k := \sup_{B_{|x_0|/2}(x_0)} |M_k v_k - v_*| \rightarrow 0 \text{ as } k \rightarrow \infty.
\]
Let $m$ denote the modulus of continuity of $v_*$ in $B_{|x_0|/2}(x_0)$, i.e.
\[
m(r) = \sup \Big\{ |v_*(x) - v_*(y)|: x, y \in B_{|x_0|/2}(x_0), |x - y| < r\Big\}.
\]

In the computation below, we use $o(1)$ to denote quantities such that $$\lim_{k \rightarrow \infty} o(1) = 0.$$

Fix some $\delta > 0$ and consider $|x - x_0| < |x_0|/2$. We note that
\[
\Big|\Big(x + \frac{(\bar\lambda_k(x_0) + \delta)^2(y_k - x)}{|y_k - x|^2}\Big) - \Big(x_0 + \frac{\bar\lambda_k(x_0)^2(y_k - x_0)}{|y_k - x_0|^2}\Big)\Big| = |x - x_0| + o(1).
\]
Thus, 
\[
\Big|v_*\Big(x + \frac{(\bar\lambda_k(x_0) + \delta)^2(y_k - x)}{|y_k - x|^2}\Big) - v_*\Big(x_0 + \frac{\bar\lambda_k(x_0)^2(y_k - x_0)}{|y_k - x_0|^2}\Big)\Big| \leq m(|x - x_0| + o(1)).
\]
It follows that
\begin{align*}
M_k (v_k)_{x, \bar\lambda_k(x_0) + \delta}(y_k)
	&= \Big(\frac{\bar\lambda_k(x_0) + \delta}{|y_k - x|}\Big)^{n-2} (M_kv_k)\Big(x + \frac{(\bar\lambda_k(x_0) + \delta)^2(y_k - x)}{|y_k - x|^2}\Big)\\
	&\geq \Big(\frac{\bar\lambda_k(x_0) + \delta}{|y_k - x|}\Big)^{n-2} \Big[(M_kv_k)\Big(x_0 + \frac{\bar\lambda_k(x_0) ^2(y_k - x_0)}{|y_k - x_0|^2}\Big)\\
		&\qquad\qquad\qquad -2\eta_k - m(|x - x_0| + o(1))\Big]\\
	&= \Big(1 + \frac{\delta}{\bar\lambda_k(x_0)}\Big)^{n-2} M_k (v_k)_{x_0, \bar\lambda_k(x_0)}(y_k)\\
		&\qquad\qquad - \Big(\frac{\bar\lambda_k(x_0) + \delta}{|y_k - x|}\Big)^{n-2}[2\eta_k + m(|x - x_0| + o(1))].
\end{align*}
Recalling \eqref{Eq:F2-1}, we arrive at
\begin{align*}
M_k (v_k)_{x, \bar\lambda_k(x_0) + \delta}(y_k)
	&\geq \Big(1 + \frac{\delta}{\bar\lambda_k(x_0)}\Big)^{n-2} M_k v_k(y_k)\\
		&\qquad\qquad - \Big(\frac{\bar\lambda_k(x_0) + \delta}{|y_k|}\Big)^{n-2}[o(1)+ m(|x - x_0| + o(1))].
\end{align*}
Thus, in view of \eqref{aadeg-0Ex}, we can find small $\bar\epsilon > 0$ depending only on $\delta$, $c$, $\bar\lambda(x_0)$ and the function $m(\cdot)$ such that, for all $|x - x_0| < \bar\epsilon$ and for large $k$,
\begin{align*}
M_k (v_k)_{x, \bar\lambda_k(x_0) + \delta}(y_k)
	&\geq \Big(1 + \frac{\delta}{4\bar\lambda_k(x_0)}\Big)^{n-2} M_k v_k(y_k).
\end{align*}
This implies that (cf. \eqref{Eq:F2-1}), that
\[
\bar\lambda_k(x) \leq \bar \lambda_k(x_0) + \delta \text{ for all } |x - x_0| < \bar\epsilon \text{ and large } k.
\]
The conclusion follows.
\end{proof}

We now return to drawing a contradiction from \eqref{Eq:Caseii}. By Lemma \ref{Lem:lamCont}, we infer from \eqref{Eq:Caseii} that there exists some $r_0 > 0$ such that $\bar\lambda(x) < |x|$ for all $x \in B_{r_0}(x_0)$. We can then argue as in the proof of Theorem \ref{proposition1}, using Lemma \ref{lem-0.3deg} instead of Lemma \ref{lem-0.3} to obtain
\[
v_*(x) = \Big(\frac{a}{1 +  b^2|x - \bar x|^2}\Big)^{\frac{n-2}{2}} \qquad x \in B_{r_0}(x_0).
\]
for some $ \bar x \in \RR^n$ and some $a, b > 0$. For small $\delta > 0$, let
\[
v_*^\delta(x) := v_*(x) + \delta |x - x_0|^2.
\]
Since $M_k v_* \rightarrow v_*$ in $C^{0}(\bar B_\delta(x_0))$, there exists $\beta_k \rightarrow 0$ and $x_k \rightarrow x_0$ such that the function $\xi_{k,\delta} := v_*^\delta + \beta_k$ satisfies
\[
(M_k v_k - \xi_{k,\delta})(x_k) = 0 \text{ and } M_k v_k - \xi_{k,\delta} \leq 0 \text{ near } x_k.
\]
It follows that
\[
A^{v_k}(x_k) \geq A^{\frac{1}{M_k} \xi_{k,\delta}}(x_k) = M_k^{\frac{4}{n-2}} A^{\xi_{k,\delta}}(x_k) .
\]
On the other hand, by hypothesis, there is some $\lambda_* \in \Gamma$ such that $f(\lambda_*) = 1$ (e.g. $\lambda_* = \lambda(A^{v_1}(0))$). By \eqref{0304weak}, we can find $\hat\lambda_* \in \Gamma$ such that $f(\hat \lambda_*) > 1$. As $M_k \rightarrow \infty$ and $A^{\xi_{k,\delta}}(x_k) = 2b^{2}a^{-2}I + O(\delta)$, we can find $k$ sufficiently large such that $M_k^{\frac{4}{n-2}} A^{\xi_{k,\delta}}(x_k) > \diag(\hat\lambda_*)$. We are thus led to
\[
A^{v_k}(x_k) > \diag(\hat\lambda_*).
\]
As $f(\lambda(A^{v_k})) = 1$ and $f(\hat \lambda_*) > 1$, the above contradicts \eqref{02weakX} and \eqref{0304weak}.
\end{proof}

\section{Local gradient estimates} \label{sec:ThmB}

In this section, we adapt the argument in \cite{Li09-CPAM} to prove Theorem \ref{TheoremB}.

For a locally Lipschitz function $w$ in $B_2(0)$, $0 < \alpha < 1$, $x \in B_2(0)$  and $0 < \delta < 2 - |x|$, define
\begin{align*}
[w]_{\alpha,\delta}(x) 
	&= \sup_{0 < |y - x| < \delta} \frac{|w(y) - w(x)|}{|y - x|^\alpha},\\
\delta(w,x,\alpha) 
	&= \left\{\begin{array}{ll}
	\infty & \text{ if } (2 - |x|)^\alpha\,[w]_{\alpha,2 - |x|}(x) < 1,\\
	\mu & \text{ where } 0 < \mu \leq 2 - |x| \text{ and } \mu^\alpha\,[w]_{\alpha,\mu}(x) = 1\\
		& \text{ if } (2 - |x|)^\alpha\,[w]_{\alpha,2 - |x|}(x) \geq 1.
\end{array}\right.
\end{align*}
Note that $\delta(w,x,\alpha)$ is well defined as $[w]_{\alpha,\delta}(x)$ is continuous and non-decreasing in $\delta$. The object $\delta(w,x,\alpha)$ was introduced in \cite{Li09-CPAM}. Its reciprocal $\delta(w,x,\alpha)^{-1}$ plays a role similar to that of $|\nabla w(x)|$ in performing a rescaling argument for a sequence of functions blowing up in $C^\alpha$-norms. For example, when $\delta = \delta(w,x,\alpha) < \infty$, the rescaled function $\hat w(y) := w(x + \delta y) - w(x)$ satisfies
\[
\hat w(0) = 0 \text{ and } [\hat w]_{\alpha,1}(0) = \delta^\alpha[\hat w]_{\alpha,\delta}(x) = 1.
\]

\begin{thm}
Theorem \ref{TheoremB} holds if we relax \eqref{01}-\eqref{04} to \eqref{01weak}-\eqref{Eq:ActaSIdiag}.
\end{thm}

\begin{proof} By the conformal invariance \eqref{Eq:CIProp}, it suffices to show bound $|\nabla \ln v|$ in $B_{1/4}(0)$. 

We first claim that
\begin{equation}
\sup_{x \neq y \in B_{1/2}(0)} \frac{|\ln v(x) - \ln v(y)|}{|x - y|^\alpha} \leq 
C(\Gamma, \alpha) \text{ for any } 0 < \alpha < 1
	.\label{11M16-1}
\end{equation}

Assume otherwise that \eqref{11M16-1} fails for some $0 < \alpha < 1$. Then there exist $0 < v_i \in C^2(B_2(0))$ such that $f(\lambda(A^{v_i})) = 1$ and $v_i \leq b$ in $B_2(0)$ but
\[
\sup_{x \neq y \in B_{1/2}(0)} \frac{|\ln v_i(x) - \ln v_i(y)|}{|x - y|^\alpha} \rightarrow \infty.
\]
This implies that, for any fixed $0 < r < 1/2$, 
\[
\sup_{x \in B_{1/2}(0)} [\ln v_i]_{\alpha,r}(x) \rightarrow \infty \text{ and } \inf_{x \in B_{1/2}(0)} \delta(\ln v_i, x, \alpha) \rightarrow 0.
\]
Therefore, there exists $x_i \in B_{1}(0)$, 
\[
\frac{1 - |x_i|}{\delta(\ln v_i, x_i, \alpha)} = \sup_{x \in B_{1}(0)} \frac{1 - |x|}{\delta(\ln v_i, x,\alpha)} \rightarrow \infty.
\]
Let $\sigma_i = \frac{1 - |x_i|}{2}$ and $\epsilon_i = \delta(\ln v_i, x_i, \alpha)$. Then
\begin{equation}
\frac{\sigma_i}{\epsilon_i} \rightarrow \infty, \epsilon_i \rightarrow 0,  \text{ and } \epsilon_i \leq 2\,\delta(\ln v_i,z,\alpha) \text{ for any } |z - x_i| \leq \sigma_i
	.\label{11M16-1x}
\end{equation}

We now define
\[
\hat v_i(y) = \frac{1}{v_i(x_i)}\,v_i(x_i + \epsilon_i\,y) \text{ for } |y| \leq \frac{\sigma_i}{\epsilon_i}.
\]
Then
\begin{equation}
[\ln \hat v_i]_{\alpha,1}(0) = \epsilon_i^\alpha\,[\ln v_i]_{\alpha,\epsilon_i}(x_i) = 1
	.\label{11M16-2}
\end{equation}
Also, by \eqref{11M16-1x}, for any fixed $\beta > 1$ and $|y| < \beta$, there holds
\begin{align}
[\ln \hat v_i]_{\alpha,1}(y) 
	&= \epsilon_i^\alpha\,[\ln v_i]_{\alpha,\epsilon_i}(x_i + \epsilon_i\,y)\nonumber\\
	&\leq 2^{-\alpha}\,\epsilon_i^\alpha \Big\{ \sup_{|z - (x_i + \epsilon_i y)| \leq \epsilon_i} [\ln v_i]_{\alpha,\epsilon_i/4}(z) + [\ln v_i]_{\alpha,\epsilon_i/4}(x_i + \epsilon_i\,y)\Big\}\nonumber\\
	&\leq  \sup_{|z - (x_i + \epsilon_i y)| \leq \epsilon_i} \delta(\ln v_i,z,\alpha)^\alpha\,[\ln v_i]_{\alpha,\delta(\ln v_i,z,\alpha)}(z)\nonumber\\
		&\qquad\qquad + \delta(\ln v_i,x_i + \epsilon_i\,y,\alpha)^\alpha\,[\ln v_i]_{\alpha,\delta(\ln v_i,x_i + \epsilon_i\,y,\alpha)}(x_i + \epsilon_i\,y)\nonumber\\
	&= 2
	\label{11M16-3}
\end{align}
for all sufficiently large $i$. Since $\hat v_i(0) = 1$ by definition, we deduce from \eqref{11M16-2} and \eqref{11M16-3} that
\begin{equation}
\frac{1}{C(\beta)} \leq \hat v_i(y) \leq C(\beta) \text{ for } |y| \leq \beta \text{ and all sufficiently large $i$}
	.\label{21F11-4}
\end{equation}
We can now apply Theorem \ref{SoftGEst} to obtain
\begin{equation}
|\nabla \ln \hat v_i| \leq C(\beta) \text{ in } B_{\beta/2}(0) \text{ for all sufficiently large $i$}.
	\label{21F11-7}
\end{equation}
Passing to a subsequence and recalling \eqref{11M16-1x} and \eqref{21F11-4}, we see that $\hat v_i$ converges in $C^{0,\alpha'}$ ($\alpha < \alpha' < 1$) on compact subsets of $\RR^n$ to some positive, locally Lipschitz function $v_*$.

On the other hand, if we define
\[
\bar v_i(y) = \epsilon_i^{\frac{n-2}{2}}\,v_i(x_i + \epsilon_i\,y) \text{ for } |y| \leq \frac{\sigma_i}{\epsilon_i},
\]
then by the conformal invariance \eqref{Eq:CIProp}, we have
\[
f(\lambda(A^{\bar v_i})) = 1 \text{ in } B_{\sigma_i/\epsilon_i}(0).
\]
Since $\frac{\sigma_i}{\epsilon_i} \rightarrow \infty$, $\hat v_i = M_i\,\bar v_i$ where $M_i = v_i(x_i)^{-1} \epsilon_i^{-\frac{n-2}{2}} \rightarrow \infty$ (thanks to the bound $v_i \leq b$), we then conclude from Theorem \ref{proposition1-deg} that $v_*$ is constant, namely
\[
v_* \equiv v_*(0) = \lim_{i \rightarrow \infty} \hat v_i(0) = 1.
\]
This contradicts \eqref{11M16-2}, in view of \eqref{21F11-7} and the convergence of $\hat v_i$ to $v_*$. We have proved \eqref{11M16-1}.

From \eqref{11M16-1}, we can find some universal constant $C > 1$ such that
\[
\frac{u(0)}{C} \leq u \leq C\,u(0) \text{ in } B_{1/2}(0).
\]
Applying Theorem \ref{SoftGEst} again we obtain the required gradient estimate in $B_{1/4}(0)$.
\end{proof}

\section{Fine blow-up analysis} \label{sec:T4X}

\subsection{A quantitative centered Liouville-type result}

In this subsection, we establish: 
\begin{prop}\label{prop-C16-1new}
Let $(f, \Gamma)$  satisfy 
\eqref{01weak}-\eqref{0304weak}, \eqref{Eq:ActaSIdiag}-\eqref{02weakY}, \eqref{FU-1a} and the normalization condition \eqref{F1}.
Assume that  for a sequence $R_k\to\infty$,  
 $0 < v_k \in C^2(B_{R_k})$ satisfy
\begin{equation}
f(\lambda(A^{v_k}))(y)=1,\ \ 0<v_k(y)\le v_k(0)=1,
\quad |y|\le R_k.
\label{ab1new}
\end{equation}
  Then for every 
$\epsilon>0$, there exists a constant
$\daste > 0$, depending only on $(f,\Gamma)$ and $\epsilon$, 
such that, for all sufficiently large $k$,
\begin{equation}
|v_k(y)-U(y)|\le 2 \epsilon
U(y),\qquad\forall\
|y|\le \daste R_k.
\end{equation}
\end{prop}

Recall that $U = (1 + |x|^2)^{-\frac{n-2}{2}}$, $A^U \equiv 2I$ and $f(\lambda(A^U)) = 1$ on $\RR^n$.

Proposition \ref{prop-C16-1new} is equivalent to the following proposition.

\begin{prop}\label{prop:BbUp1}
Let $(f,\Gamma)$ satisfy \eqref{01weak}-\eqref{0304weak}, \eqref{Eq:ActaSIdiag}-\eqref{02weakY}, \eqref{FU-1a} and the normalization condition \eqref{F1}. For any $\epsilon > 0$ there exist $\daste, \Caste > 0$ depending only on $(f,\Gamma)$ and $\epsilon$  such that if $0 < u \in C^2(B_{R}(0))$, $R > 0$, satisfies
\[
f(\lambda(A^u)) = 1 \text{ in } B_{R}(0) \text{ and } u(0) = \sup_{B_R(0)} u \geq \Caste\, R^{-\frac{n-2}{2}},
\]
then
\[
|u(x) - U^{0,u(0)}(x)| \leq 2\epsilon U^{0,u(0)}(x) \text{ for all } x \in B_{\daste R}(0).
\]
\end{prop}

\begin{proof}[Proof of the equivalence between Proposition \ref{prop-C16-1new} and Proposition \ref{prop:BbUp1}] It is clear that Proposition \ref{prop:BbUp1} implies Proposition \ref{prop-C16-1new}.

Consider the converse. Let $\daste = \daste(\epsilon)$ be as in Proposition \ref{prop-C16-1new}. Arguing by contradiction, we assume that there are some $\epsilon > 0$ and a sequence of $R_k$ and $u_k \in C^2(B_{R_k}(0))$ such that 
\[
f(\lambda(A^{u_k})) = 1 \text{ in } B_{R_k}(0) \text{ and } u_k(0) = \sup_{B_{R_k}(0)} u_k \geq k\,R_k^{-\frac{n-2}{2}}
\]
but the last estimate in Proposition \ref{prop:BbUp1} fails for each $k$.

Define
\[
\bar u_k(y) = \frac{1}{u_k(0)} u_k\Big(\frac{y}{u_k(0)^{\frac{2}{n-2}}}\Big) \text{ for } |y| \leq R_k\,u_k(0)^{\frac{2}{n-2}} =: \bar R_k.
\]
Then $f(\lambda(A^{\bar u_k})) = 1$ in $B_{\bar R_k}(0)$, $\sup_{B_{\bar R_k}(0)} \bar u_k = \bar u_k(0) = 1$, and $\bar R_k \geq k^{\frac{2}{n-2}} \rightarrow \infty$. By Proposition \ref{prop-C16-1new},
\[
|\bar u_k(y) - U(y)| \leq 2\epsilon U(y) \text{ in } B_{\daste\,\bar R_k}(0) \text{ for all sufficiently large }k.
\]
Returning to the original sequence $u_k$ we arrive at a contradiction.
\end{proof}

\begin{lem}\label{lem0.1}
Under the hypotheses of Proposition \ref{prop-C16-1new}
except for \eqref{FU-1a}, we have
\begin{equation}
v_k\to U,\qquad \mbox{in}\ C^\beta_{loc}(\RR^n),\ \ \forall\ 0<\beta<1.
\label{limit}
\end{equation}
Moreover,  
 for  every $\epsilon>0$, there exists   $k_0\ge 1$
such that
\begin{equation}
\min_{|y|=r}v_k(y)\le (1+\epsilon)U(r),\qquad\forall\  0<r<
 R_k/5, \ k\ge k_0.
\label{eq35}
\end{equation}
\end{lem}

\begin{proof}  We first prove 
\eqref{limit}.  Since $v_k$ satisfies \eqref{ab1new},
  we deduce from 
Theorem \ref{TheoremB} that
$$
|\nabla \ln v_k|\le C\quad\mbox{in}\ B_{ R_k-1},\ \ \forall\ k,
$$
where $C$ is independent of $k$.
This yields \eqref{limit} in view of Theorem \ref{proposition1}.

We now prove \eqref{eq35}.
Suppose the contrary, then there exists some $\epsilon>0$ and  sequences of
$k_i\to \infty$, $0<r_i<R_{k_i}/5$ such that
\begin{equation}
v_{k_i}> (1+\epsilon)U\ \ \ \mbox{on}\ \partial B_{ r_i}.
\label{ri}
\end{equation}
Because of \eqref{limit},
 $r_i\to\infty$.

As in the proof of Lemma \ref{lem-1new},
there exists $\lambda^{(0)}_i>0$ such that
\begin{equation}
(v_{k_i})_\lambda \le v_{k_i}\
\mbox{in}\ B_{r_i}\setminus B_\lambda, \forall 0<\lambda<
\lambda^{(0)}_i\
\mbox{and}\  |x|\le r_i.
\label{26}
\end{equation}
By the explicit expression of $U$, there exists some small $\delta>0$ independent of $i$
such that, for large $i$, 
$$
U_\lambda(y)\le (1+\frac \epsilon 4)U(y),\quad
\forall\ y\in \partial B_{ r_i},\ 
\lambda^{(0)}_i \le \lambda\le 1+\delta,
$$
By the uniform convergence of $v_{k_i}$ to $U$ on compact
subsets of $\RR^n$, we have, for large $i$,
$$
(v_{k_i})_\lambda \le (1+\frac \epsilon 2) 
U(y),\quad
\forall\ y\in \partial B_{ r_i},\
\lambda^{(0)}_i \le \lambda\le 1+\delta,
$$
As in the proof of Lemma \ref{lem-1new}, the moving sphere 
procedure does not stop before reaching $\lambda=1+\delta$, namely
we have, for large $i$, 
$$
(v_{k_i})_\lambda \le v_{k_i}\
\mbox{in}\ B_{r_i}\setminus B_\lambda, \forall 0<\lambda<
1+\delta\
\mbox{and}\  |x|\le r_i.
$$
Sending $i$ to $\infty$ leads to
$$
U_{1+\delta}(y)\le U(y),\ \ \forall\ 1+\delta\le |y|\le 2.
$$
A contradiction --- since we see from the explicit expression of $U$ that
$U_{1+\delta}(y)>U(y)$
for all $1<1+\delta< |y|\le 2$.
\end{proof}

\begin{lem}\label{lem-energy1}
Under the hypotheses of Proposition \ref{prop-C16-1new}, 
 for  any $\epsilon > 0$, there exist a small $\delta_1>0$ and a large  $r_1 > 1$,
depending only on $(f, \Gamma)$ and $\epsilon$,
such that, for all sufficiently large $k$,
\begin{align} 
&v_k(y) \ge (1-\epsilon)U(y),
\qquad\forall\ |y|\le \delta_1 R_k,
\label{U00}\\
\text{ and }\qquad &\int_{ r_1\le |y|\le \delta_1 R_k }
v_k^{\frac {n+2}{n-2}}
\le \epsilon.
\label{energy1}
\end{align}
\end{lem}

\begin{proof} Assume without loss of generality that $\epsilon \in (0,1/2)$. Since
$v_k\to U$ in $C^0_{loc}(\RR^n)$, 
there exist $r_2>1$ and $k_1$,  depending on $\epsilon$, 
 such that for all $k\ge k_1$
\begin{align}
v_k(y)&\ge (1-\epsilon^2)U(y),\qquad
\forall\ |y|\le r_2,
\label{U10}\\
v_k(y)&\ge (1-\epsilon^2)U(r_2)\ge (1-2\epsilon^2) r_2^{2-n},
\qquad\forall\ |y|=r_2.
\label{U0}
\end{align}

By \eqref{FU-1a},
$$
Trace\ (A^{v_k})\ge \delta > 0,
$$
and therefore
\begin{equation}
-\Delta v_k(y)\ge \frac {n-2}2 \delta v_k(y) ^{ \frac {n+2}{n-2} }
\qquad\mbox{in}\  r_2\le 
 |y| \le R_k.
\label{U4}
\end{equation}

Using the superharmonicity of $v_k$
and the maximum principle, we obtain
$$
v_k(y)\ge 
(1-\epsilon^2)
\left( |y|^{2-n}-
R_k^{2-n}\right),\qquad r_2\le |y|\le R_k.
$$
Thus, for any $\delta_2 \in (0,\epsilon^{\frac{2}{n-2}})$, we have for all sufficiently large $k$ that
\begin{equation}
v_k(y)\ge 
(1-\epsilon^2)(1 - \delta_2^{n-2}) 
|y|^{2-n} \geq (1-2\epsilon^2)
|y|^{2-n},\qquad
r_2\le |y|\le \delta_2 R_k.
\label{U5}
\end{equation}
Now if $\delta_1 < \delta_2$, \eqref{U00} is readily seen from \eqref{U10} and \eqref{U5}.

Let
$$
\hat v_k(y):= v_k(y)- (1-2\epsilon^2)
|y|^{2-n}.
$$
Then
$$
-\Delta 
\hat v_k(y)\geq \hat f(y):=
\frac {n-2}2 \delta v_k(y) ^{ \frac {n+2}{n-2} }
\qquad\mbox{in}\  r_2\le 
 |y| \le \delta_2 R_k,
$$
and
$$
\hat v_k(y)\ge 0,\quad \mbox{for}\ y\in \partial (B_{\delta_2 R_k}\setminus 
B_{ r_2}).
$$
Let $R_k' = \frac{\delta_2R_k}{2}$. Enlarging $k_1$ if necessary, we can apply Corollary \ref{cor-App1-4} 
in Appendix \ref{Sec:AppA} to get
\begin{equation}
\min_{ |x|= R_k'}
 \hat v_k(x)
\ge 
C^{-1} (\delta_2 R_k)^{2-n}
\int_{ 2r_2\le |y|\le \delta_2 R_k/8} \delta  v_k(y) ^{ \frac {n+2}{n-2} }
dy, \qquad
\forall\ k \geq k_1,
\label{U6}
\end{equation}
where here and below $C$ is some positive constant depending only on $n$. On the other hand, by Lemma \ref{lem0.1}, we have (after enlarging $k_1$ if necessary) 
\[
\min_{ |x|= R_k'} v_k(x) \le(1+\epsilon^2) U(R_k')
\le   (1+2\epsilon^2)(R_k')^{2-n}, \qquad  \forall\ k \geq k_1,
\label{U7}
\]
which implies that
\begin{equation}
\min_{ |x|= R_k'} \hat v_k(x) \le C\,\epsilon^2 (\delta_2 R_k)^{2-n}, \qquad  \forall\ k \geq k_1.
\label{U7}
\end{equation} 
It now follows from \eqref{U6} and \eqref{U7} that
$$
\int_{ 2r_2\le |y|\le \delta_2 R_k/8} \delta  v_k(y) ^{ \frac {n+2}{n-2} }
dy
\le 
c_1\,\epsilon^2
$$
where $c_1$ depends only on $n$. \eqref{energy1} is then established for $\epsilon \leq \frac{1}{c_1}$ with $r_1 = 2r_2$ and $\delta_1 = \delta_2/8$. The conclusion for $\epsilon > 1/c_1$ also follows.
\end{proof}

\begin{lem}  Let $(f, \Gamma)$ satisfy \eqref{01weak}-\eqref{0304weak}. 
Then there  exist $\delta_3>0$ and $C_3>1$, depending only on $(f, \Gamma)$, such that if $u\in
C^2(B_2(0))$ satisfies
$$
f(\lambda(A^u))=1,   u>0,  \qquad \mbox{in}\ B_2(0),
$$
and
$$
\int_{ B_2(0)} u^{ \frac {2n}{n-2} }\le \delta_3,
$$
then 
$$
u\le C_3\quad\mbox{in}\ B_1(0).
$$
\label{lem-smallenergy}

If $(f,\Gamma)$ satisfies in addition the conditions \eqref{Eq:ActaSIdiag}, \eqref{02weakY} and the normalization condition \eqref{F1}, then $\delta_3$ can be chosen to be any constant smaller than $\int_{\RR^n} U^{\frac{2n}{n-2}}\,dx$.
\end{lem}

\begin{proof}
We adapt the proof of \cite[Lemma 6.4]{LiLi03}. Arguing by contradiction, we can find a sequence of $0 < u_j \in C^2(B_2)$ such that $f(\lambda(A^{u_j})) = 1$ in $B_2(0)$,
\[
\int_{B_2(0)} u_j ^{ \frac {2n}{n-2} } \rightarrow 0
\]
but
\[
d(y_j)^{\frac{n-2}{2}}u_j(y_j) = \max_{\bar B_{3/2}(0)} d(y)^{\frac{n-2}{2}}u_j(y) \rightarrow \infty,
\]
where $y_j \in B_{3/2}(0)$ and $d(y) = 3/2 - |y|$.

Let $\sigma_j = \frac{1}{2}d(y_j) > 0$,
\[
v_j(z) = \frac{1}{u_j(y_j)}u_j\Big(y_j + \frac{1}{u_j(y_j)^{\frac{2}{n-2}}} z\Big) \text{ for } |z| < r_j := u_j(y_j)^{\frac{2}{n-2}} \sigma_j \rightarrow \infty.
\]
Then by the conformal invariance property \eqref{Eq:CIProp}, $f(\lambda(A^{v_j})) = 1$ in $B_{r_j}(0)$, $v_j(0) = 1$, $v_j \leq 2^{\frac{n-2}{2}}$ in $B_{r_j}(0)$ and
\begin{equation}
\int_{B_{r_j}(0)} v_j^{\frac{2n}{n-2}} \rightarrow 0.
	\label{Eq:SLcrit}
\end{equation}

By Theorem \ref{TheoremB}, there is a constant $C$ independent of $j$ such that
\[
|\nabla \ln v_j| \leq C \text{ in } B_{r_j/2}(0).
\]
Thus, after passing to a subsequence, we can assume that $v_j$ converges in $C^0_{loc}(\RR^n)$ to some positive function $v$ (as $v_j(0) = 1$). This contradicts \eqref{Eq:SLcrit}.

The above argument can be adapted to prove the last assertion of the lemma: Equation \eqref{Eq:SLcrit} is replaced by
\[
\int_{B_{r_j}(0)} v_j^{\frac{2n}{n-2}} \leq \delta_3 <  \int_{\RR^n} U^{\frac{2n}{n-2}}\,dx.
\]
On the other hand, by Theorem \ref{proposition1}, we have $v_j \rightarrow U$ in $C^0_{loc}(\RR^n)$. This gives a contradiction.
\end{proof}

\begin{lem}\label{lem:Ressmallenergy}
\  Let $(f, \Gamma)$ satisfy \eqref{01weak}-\eqref{0304weak} and let $\delta_3$, $C_3$ be as in Lemma \ref{lem-smallenergy}. If $u\in
C^2(B_{2R}(0))$ satisfies
$$
f(\lambda(A^u))=1,   u>0,  \qquad \mbox{in}\ B_{2R}(0),
$$
and
$$
\int_{ B_{2R(0)}} u^{ \frac {2n}{n-2} }\le \delta_3,
$$
then 
$$
u\le C_3\,R^{-\frac{n-2}{2}}\quad\mbox{in}\ B_{R}(0).
$$
\end{lem}

\begin{proof} This follows from Lemma \ref{lem-smallenergy} and a change of variables, $\tilde u(y) = R^{\frac{n-2}{2}}u(Ry)$ for $|y| \leq 2$.
\end{proof}

\bigskip

\begin{lem}\label{lem-upperbound}
Under the hypotheses of Proposition \ref{prop-C16-1new}, there exist positive constants $\delta_4 > 0$ and $C_4 > 1$, depending only on $(f,\Gamma)$,  such that, for all sufficiently large $k$,
\begin{equation}
v_k(y)\le C_4 U(y),
\qquad\forall\
|y|\le \delta_4 R_k.
\end{equation}
\end{lem}

\begin{proof} Let $\delta_3$ be as in Lemma \ref{lem-smallenergy}. Since $v_k \leq 1$, we deduce from Lemma \ref{lem-energy1}, there is $r_1 > 1$ and $\delta_1 > 0$ such that
\begin{equation}
\int_{ r_1\le |y|\le \delta_1 R_k} v_k^{ \frac {2n}{n-2} }
\le  \epsilon.
\label{energy2}
\end{equation}

For any $2r_1<r< \delta_1 R_k/2$, consider
$$
\tilde v_k(z)= r^{ \frac {n-2}2 } v_k(rz),\quad
\frac 12<|z|<2.
$$
By \eqref{energy2}, we have, for large $k$,   
$$
\int_{ \frac 12<|z|<2 } \tilde v_k(z)^{ \frac {2n}{n-2} }
=\int_{  \frac r2 <|\eta|< 2r}
v_k(\eta)^{ \frac {2n}{n-2} }
\le  \delta_3.
$$
It follows from Lemma \ref{lem:Ressmallenergy} that
$$
\tilde v_k(z)\le C \qquad \forall\ \frac 23 <|z|<  \frac 54,
$$
for some universal constant $C$. Since $\tilde v_k$ also satisfies $f(\lambda(A^{\tilde v_k})) = 1$, we can apply Theorem \ref{TheoremB} to obtain
\[
|\nabla \ln \tilde v_k(z)| \leq C \qquad \forall\  |z| = 1,
\]
which implies that $\max_{|z| = 1} \tilde v_k \leq C\,\min_{\partial B_1}\tilde v_k$. Returning to $v_k$, we obtain
\[
\max_{\partial B_r} v_k \leq C\,\min_{\partial B_r} v_k 
\]
where $C$ is universal. The conclusion then follows from Lemma \ref{lem0.1}.
\end{proof}

\begin{proof}[Proof of Proposition \ref{prop-C16-1new}]
Fix $\epsilon > 0$. In view of Lemma \ref{lem-energy1} (cf. \eqref{U00}), we only need to prove that there exist $\daste > 0$ such that, for all sufficiently large $k$,
\begin{equation}
v_k(y)\le (1+2\epsilon)U(y), \qquad \forall\ |y|\le \daste R_k.
\label{C17-1}
\end{equation}
Suppose the contrary of the above, then, after passing to a subsequence 
and renaming the subsequence still as $\{v_k\}$ and $\{R_k\}$,
there exist $|y_k|=\delta_k R_k$, $\delta_k\to 0^+$,
such that
\begin{equation}
v_k(y_k)=\max_{ |y|=\delta_k R_k}
v_k(y)> (1+2\epsilon) U(y_k).
\label{C17-3}
\end{equation}
In view of the convergence of $v_k$ to $U$,
 $|y_k|\to \infty$ as $k\to \infty$.

Consider the following two rescalings of $v_k$:
\begin{equation}
\hat v_k(z):= |y_k|^{n-2} 
v_k(|y_k|z) \text{ and } \bar v_k(z) = |y_k|^{\frac{n-2}{2}}v_k(|y_k|z),\qquad
|z|<\frac {  R_k}{ |y_k|}\to \infty.
\label{C18-2}
\end{equation}
By Lemma \ref{lem-upperbound}, we have
\begin{equation}
\hat v_k(z)\le C|z|^{2-n} \text{ and } \bar v_k(z) \leq C\,|y_k|^{-\frac{n-2}{2}} |z|^{2-n}
	\label{Eq:hatbarvkUB}
\end{equation}
for some constant $C$ independent of $k$.

In view of the conformal invariance \eqref{Eq:CIProp} and \eqref{ab1new}, 
\[
f(\lambda(A^{\bar v_k}(z))) = 1 \text{ for } |z|< \frac {  R_k}{ |y_k|}.
\]
Recalling \eqref{Eq:hatbarvkUB}, we can apply Theorem \ref{TheoremB} to obtain that
for all $0<\alpha<\beta <\infty$, there exists
positive constant $C(\alpha, \beta)$ such that
for large $k$,
\begin{equation}
|\nabla \ln \bar v_k(z)|\le C(\alpha, \beta),\qquad
\forall \ \alpha<|z|<\beta,
\label{C22-2X}
\end{equation}
which implies that
\begin{equation}
|\nabla \ln \hat v_k(z)|\le C(\alpha, \beta),\qquad
\forall \ \alpha<|z|<\beta.
\label{C22-2}
\end{equation}

We know from \eqref{C18-2},  \eqref{C17-3} and
Lemma \ref{lem0.1} that
\begin{equation}
\min_{ |z|=1} \hat v_k(z)\le 
(1+ \epsilon)\frac{|y_k|^{n-2}}{U(y_k)},
\label{C19-2}
\end{equation}
and
\begin{equation}
\max_{ |z|=1} \hat v_k(z)\ge (1+ 2\epsilon)\frac{|y_k|^{n-2}}{U(y_k)}.
\label{C20-1}
\end{equation}

We deduce from \eqref{C22-2}, \eqref{C19-2} and \eqref{C20-1}, after passing to a subsequence,
that for some positive function
$\hat v^*$ in $C^{0,1}_{loc}(\RR^n\setminus\{0\})$,
\begin{equation}
\hat v_k\to \hat v^*\qquad \mbox{in}\ C^{\alpha}_{loc}(\RR^n\setminus
\{0\}),\ \forall\ 0<\alpha<1.
\label{C23-1}
\end{equation}
By Theorem \ref{proposition1-deg}, $\hat v^*$ is radially symmetric.
On the other hand, we  deduce from \eqref{C19-2} and
\eqref{C20-1} after passing to limit that
\begin{equation}
\min_{ |z|=1} \hat v^*(z)\le 
1+\epsilon,
\quad\mbox{and}\quad
\max_{ |z|=1} \hat v^*(z)\ge 1+2\epsilon.
\label{C20-1new}
\end{equation}
The above violates the radial symmetry of $\hat v^*$.
 Proposition \ref{prop-C16-1new} is established.
\end{proof}

\subsection{Detailed blow-up landscape}

The proof of Theorem \ref{theorem4X} uses the following consequence of the Harnack-type inequality for conformally invariant equations, see \cite{SchoenNotes, ChenLin, LiLi03}. 

\begin{lem}\label{Lem:EBnd}
Let $(f,\Gamma)$ satisfy \eqref{01weak}-\eqref{0304weak} and \eqref{FU-1a}. There exists a constant $C_6$, depending only on $(f,\Gamma)$, such that if $u \in C^2(B_{3}(0))$ is a positive solution of
\[
f(\lambda(A^u)) = 1 \text{ in } B_{3}(0)
\]
then 
\[
\int_{B_1(0)} |u|^{\frac{2n}{n-2}}\,dx \leq C_6.
\]
\end{lem}

\begin{proof} We give the proof here for completeness. By \eqref{FU-1a}, 
\[
-\Delta u \geq \frac{n-2}{2}\delta\,u^{\frac{n+2}{n-2}} > 0 \text{ in } B_2(0).
\]
Thus, by Corollary \ref{cor-App1-3X} in Appendix \ref{Sec:AppA} as well as the maximum principle,
\[
\inf_{B_{2}(0)} u = \inf_{B_{2}(0) \setminus B_{3/2}(0)} u \geq \frac{1}{C} \int_{B_1(0)} u^{\frac{n+2}{n-2}}\,dx.
\]
It follows that
\[
\int_{B_1(0)} u^{\frac{2n}{n-2}}\,dx \leq C\sup_{B_1(0)} u \,\inf_{B_{2}(0)} u.
\]
The conclusion follows from the above estimate and the Harnack-type inequality \cite[Theorem 1.2]{LiLi05}. (Note that \eqref{FU-1a} is used again here.)
\end{proof}

\begin{proof}[Proof of Theorem \ref{theorem4X}]
In view of Proposition \ref{prop:BbUp1} and (vi), it suffices to establish the theorem for $\epsilon = \epsilon_0 := 1/2$.

By Lemma \ref{Lem:EBnd},
\begin{equation}
\int_{B_2(0)} u^{\frac{2n}{n-2}} \,dx \leq C_6.
	\label{Eq:EBnd}
\end{equation}
The constant $\bar m$ in the result can be selected to be the least integer satisfying
\[
\bar m \geq 2C_6 \Big(\int_{B_1} U^{\frac{2n}{n-2}}\,dx\Big)^{-1}.
\]
(Clearly, this is an obvious upper bound for $m$ if the $x^i$'s satisfies (iii).)

Let $\delta_3$ and $C_3$ be the constants in Lemma \ref{lem:Ressmallenergy}. Fix some $N_0 > \frac{C_1}{\delta_3}$. 
Then there is some $r_0 \in (3/2, 2)$ such that
\[
\int_{r_0 < |x| < r_0 + \frac{1}{2N_0} } u^{\frac{2n}{n-2}}\,dx \leq  \delta_3.
\]
By Lemma \ref{lem:Ressmallenergy}, this implies that
\begin{equation}
u(x) \leq C_3\,(8N_0)^{\frac{n-2}{2}} =: C_{7} \text{ for all } r_0 + \frac{1}{8N_0} < |x| < r_0 + \frac{3}{8N_0}.
	\label{Eq:G0Barrier}
\end{equation}

Let $\Caste$ and $\daste$ be as in Proposition \ref{prop:BbUp1} (corresponding to $\epsilon = \epsilon_0$). We can assume without loss of generality that 
\begin{equation}
\Caste > 2 \text{ and }\daste < 1.
\label{Eq:PreRC*d*}
\end{equation}
We now declare
\begin{equation}
\Caspr = \max\Big(2C_7, \Caste (2\daste^{-1})^{\frac{\bar m(n-2)}{2}}(4N_0)^{-\frac{n-2}{2}}\Big).
	\label{Eq:C*'Choice}
\end{equation}
This choice of $\Caspr$ will become clear momentarily.

Let $U_1 = B_{r_0 + \frac{3}{8N_0}}(0)$ and $V_1 = B_{r_0 + \frac{1}{8N_0}}(0) \subset U_0$. By \eqref{Eq:C*'Choice}, $\Caspr \geq 2C_7$, and so, by \eqref{Eq:G0Barrier}, there is some $x^1 \in V_1 \subset B_2(0)$ such that
\[
u(x^1) = \sup_{U_1} u \geq \Caspr.
\]

Let $R_1 = \frac{1}{4N_0}$, then \eqref{Eq:C*'Choice} gives
\[
\Caspr \geq \Caste\,R_1^{-\frac{n-2}{2}}.
\]
Hence, an application of Proposition \ref{prop:BbUp1} to $u$ on the ball $B_{R_1}(x^1)$ leads to
\[
|u(x) - U^{x^1, u(x^1)}(x)|
\le \epsilon_0\,U^{x^1, u(x^1)}(x),
\ \ \forall \ x\in B_{\daste R_1}(x^1).
\]
In particular, for $\frac{\daste R_1}{2} \leq |x - x^1| \leq \daste\,R_1$, 
\begin{equation}
u(x) \leq 2U^{x^1, u(x^1)}(x) \leq \frac{2}{u(x^1) |x - x^1|^{n-2}} \leq \frac{2^{n-1}}{\Caspr (\daste R_1)^{n-2}} \leq \frac{\Caspr}{2},
	\label{Eq:G1Barrier}
\end{equation}
where we have used \eqref{Eq:C*'Choice} in the last estimate.

Let $U_2 = U_1 \setminus B_{\daste\,R_1/2}(x^1)$ and $V_2 = V_1 \setminus B_{\daste R_1}(x^1) \subset U_1$. If 
\[
\sup_{U_2} u \leq \Caspr,
\]
we stop. Otherwise, in view of \eqref{Eq:G1Barrier}, there is some $x^2 \in V_2$ such that
\[
u(x^2) = \sup_{U_2} u \geq \Caspr.
\]
We then let $R_2 = \frac{\daste R_1}{2}$ so that \eqref{Eq:C*'Choice} implies
\[
\Caspr \geq \Caste\,R_2^{-\frac{n-2}{2}}.
\] 
Hence, by Proposition \ref{prop:BbUp1},
\[
|u(x) - U^{x^2, u(x^2)}(x)|
\le \epsilon_0 U^{x^2, u(x^2)}(x),
\ \ \forall \ x\in B_{\daste R_2}(x^2).
\]
We then repeat the above process to define $U_3$, $V_3$, and to decide if a local max $x^3$ can be selected in $V_3$, etc. As explain above, the number $m$ of times this process can be repeated cannot exceed $\bar m$.

We have obtained the set of local maximum points $\{x^1, \cdots, x^m\}$ of $u$ and have verified (i) and (iv) for
\[
\daspr =  \Big(\frac{\daste}{2}\Big)^{\bar m} \frac{1}{2N_0} \leq \daste\,R_{m}.
\]
(vi) is readily seen as
\[
dist(x^i, \partial U_i) \geq R_i \geq \daspr.
\]
(ii) is also clear for
\[
\Cfive \geq  \Big(\frac{2}{\daste}\Big)^{\bar m} 4N_0 \geq \frac{2}{\daste\,R_{m-1}}.
\]
From construction, we have
\[
\sup_{U_{m+1}} u \leq \Caspr.
\]
By Theorem \ref{TheoremB}, this implies that
\begin{equation}
|\nabla \ln u(x)| \leq C_8 \text{ for all } x \in V_m = V_0 \setminus \cup_{i = 1}^m B_{\daste R_i}(x^i).
	\label{Eq:HarOutsideCore}
\end{equation}
Also, note that, for $\daspr < |x - x^i| < \daste R_i$, we have
\[
\frac{1}{u(x^i)} \Big(\frac{1}{(\Caspr)^{-\frac{4}{n-2}} + (\daste\,R_i)^2.}\Big)^{\frac{n-2}{2}}
	\leq U^{x^i, u(x^i)}(x) \leq \frac{1}{u(x^i)} (\daspr)^{-(n-2)} , 
\]
and so 
\begin{equation}
\frac{1}{C_9u(x^i)} \leq u(x) \leq \frac{C_9}{u(x^i)} (\Caspr)^{-2}\,(\daspr)^{-(n-2)}
	\label{Eq:HarMargin}
\end{equation}
It is now clear that (iii) and (v) hold for $\Cfive$ sufficiently large. The proof is complete.
\end{proof}

\section{A quantitative Liouville theorem}\label{sec:tQL}

\begin{proof}[Proof of Theorem \ref{quantitativeliouville}]
Assume by contradiction that, for some $\epsilon \in (0,1/2]$, there exist $v_k \in C^2(B_{3R_k}(0)$, $R_k \rightarrow \infty$, such that $f(\lambda(A^{v_k})) = 1$ in $B_{3R_k}(0)$ and $v_k \geq \gamma$ in $B_{r_1}(0)$ but, for each $k$,
\begin{equation}
\text{\eqref{AppB3} and \eqref{AppB4} cannot simultaneously hold for any $\bar x$.}
\label{Eq:QLCAs}
\end{equation}

Define
\[
u_k(y) = R_k^{\frac{n-2}{2}}v_k(R_k\,y) \text{ for } |y| \leq 3.
\]
Then $f(\lambda(A^{u_k})) = 1$ in $B_3(0)$ and
\begin{equation}
u_k \geq R_k^{\frac{n-2}{2}} \gamma \text{ in } B_{r_1/R_k}(0). 
	\label{Eq:ukNondeg}
\end{equation}
Thus, by applying Theorem \ref{theorem4} and after passing to a subsequence, we can select sets of local maximum points $\{x_k^1, \cdots, x_k^m\}$ of $u_k$ such that assertions (i)-(vi) in Theorem \ref{theorem4} hold. We can also assume that $x_k^i \rightarrow x_*^i$, $1 \leq i \leq m$.

By assertions (i), (iv) and (v) of Theorem \ref{theorem4}, $u_k$ converges locally uniformly to zero in $B_1 \setminus \{x_*^1, \cdots, x_*^m\}$. Thus, in view of \eqref{Eq:ukNondeg}, we must have $x_*^{i_0} = 0$ for some (unique) $1 \leq i_0 \leq m$. Clearly, $B_{r_1/R_k}(0) \subset B_{\daspr}(x_k^{i_0})$ for large $k$.

Recalling assertions (iv), (vi) and returning to the original sequence $v_k$ we get, for $\bar x_k = R_k\,x_k^{i_0}$ and $\bar\mu_k = v_k(\bar x_k) = \sup_{B_{\daspr R_k}(\bar x_k)} v_k$, that $B_{r_1}(0) \subset B_{\daspr R_k}(\bar x_k)$, $\bar\mu_k \geq \gamma$ and
\[
(1 - \epsilon) U^{\bar x_k, \bar\mu_k} \leq v_k \leq (1 + \epsilon) U^{\bar x_k, \bar\mu_k} \text{ in } B_{\daspr R_k}(\bar x_k).
\]
We then have
\[
\gamma \leq v_k(0) = 2\Big(\frac{\bar \mu_k^{\frac{2}{n-2}}}{1 + \bar\mu_k^{\frac{4}{n-2}}|\bar x_k|^2}\Big)^{\frac{n-2}{2}}
\leq 
\frac{2}{\bar\mu_k\,|\bar x_k|^{n-2}}  \leq \frac{2}{\gamma\,|\bar x_k|^{n-2}} 
\]
This implies 
\[
|\bar x_k| \leq 2^{\frac{1}{n-2}}\,\gamma^{-\frac{2}{n-2}}.
\]
On the other hand, since $B_{r_1}(0) \setminus B_{r_1/2}(\bar x_k) \neq \emptyset$, we can select some $y_k \in B_{r_1}(0)$ such that
\[
|\bar x_k - y_k| \geq \frac{r_1}{2}.
\]
This implies that
\[
\gamma \leq v(y_k) \leq 2U^{\bar x_k, \bar\mu_k}(y_k) \leq \frac{2}{\bar\mu_k\,|y_k - \bar x_k|^{n-2}} \leq \frac{2^{n-1}}{\bar \mu_k\,r_1^{n-2}},
\]
and so
\[
\bar\mu_k \leq \frac{2^{n-1}}{\gamma\,r_1^{n-2}}.
\]
We have thus shown that $\bar x = \bar x_k$ satisfies both \eqref{AppB3} and \eqref{AppB4}, which contradicts \eqref{Eq:QLCAs}.
\end{proof}

\appendix

\section{A remark on positive superharmonic functions}
\label{Sec:AppA}

Let $G(x,y)$ be the Green's function of $-\Delta$ in $B_1\setminus B_\rho\subset \RR^n$,  $0<\rho <1/2$:
$$
G(x,y):= \frac 1{ n(n-2) \alpha(n) }
\left(  \frac 1{  |x-y|^{n-2} }
-h(x,y)\right), \qquad x,y\in B_1\setminus B_\rho,,
$$
where
$\alpha(n)$ denotes the
volume of the unit ball in $\RR^n$,
and
 $h(x,y)$ satisfies, for $x\in B_1\setminus B_\rho$,
$$
\left\{
\begin{array}{rlr}
-\Delta_y h(x,y)=& 0, & y\in B_1\setminus B_\rho,\\
h(x,y)=& \frac 1{  |x-y|^{n-2} }, & y\in \partial (B_1\setminus B_\rho).
\end{array}
\right.
$$

\begin{lem}\label{lem-App1}  For any $0<2\rho< \rho_0<\rho_1<\rho_2<1$,
there exists some constants $C, C'>1$, depending only on $n, \rho_0,
\rho_1, \rho_2$, such that the Green's function $G$
for $B_1\setminus B_\rho$ satisfies
$$
G(x, y)\ge  \frac 1C \left(  1-\frac {\rho^{n-2}}  { |y|^{n-2} }\right),
\qquad \forall\ \rho\le |y|\le \rho_0,
\quad  \rho_1\le |x|\le \rho_2.
$$
Consequently,
$$
G(x,y)\ge 1/C',\qquad \forall\ 2\rho\le |y|\le \rho_0, \rho_1\le |x|\le \rho_2.
$$
\end{lem}

\begin{proof} In the following we use $C_1, C_2, ...$
to denote positive constants
depending only on $\rho_0, \rho_1, \rho_2$ and $n$. 
For a fixed $x$ satisfying $\rho_1\le |x|\le \rho_2$,
it follows from the maximum principle that
for  some positive constant $C_1$, 
$$
0<h(x, y)\le C_1, \quad \forall y\in B_1\setminus B_\rho.
$$ 
It follows that for some positive constants $C_2$ and $C_3$,
$$
G(x,y)\ge |x-y|^{2-n}-C_1\ge 1/C_2,\qquad
\forall\ y\in B_{ 2/C_3}(x)\setminus \{x\}
\subset B_1\setminus B_{\rho_0}.
$$ 
Since
 $G(x, y)$ is 
a positive harmonic function of $y$ in 
$(B_1\setminus B_\rho)\setminus \{x\}$.
We can apply the Harnack inequality to obtain, for some $C_4$,
$$
G(x,y)\ge 1/C_4,\qquad \forall\   |y|=\rho_0.
$$
By the maximum principle,
$$
G(x,y)\ge 
\frac 1{C_4} \left(  1-\frac {\rho^{n-2}}  { |y|^{n-2} }\right),
\qquad \rho<|y|\le \rho_0.
$$
It follows that for some $\Cfive$,
$$
G(x,y)\ge 1/\Cfive,\qquad 2\rho\le |y|\le \rho_0.
$$
Lemma \ref{lem-App1} is established.
\end{proof}

\begin{cor}  For any $0<2\rho< \rho_0<\rho_1<\rho_2<1$,
let
 $$
\left\{
\begin{array}{rl}
-\Delta \tilde v= \tilde f\ge & 0,\qquad B_1\setminus B_\rho,\\
\tilde v\ge &0, \qquad \partial (B_1\setminus B_\rho).
\end{array}
\right.
$$
Then, for  some constants $C, C'>1$ depending only on $n, \rho_0,
\rho_1, \rho_2$, 
$$
\inf_{ \rho_1\le |x|\le \rho_2 }
\tilde v(x)\ge 
\frac 1C \int_{ \rho\le |y|\le \rho_0}
 \left(  1-\frac {\rho^{n-2}}  { |y|^{n-2} }\right)
\tilde f(y)dy
\ge \frac 1{C'}
 \int_{ 2\rho\le |y|\le \rho_0}
\tilde f(y)dy.
$$
\label{cor-App1-3}
\end{cor}

\begin{proof}
For $\rho_1\le |x|\le \rho_2$,
we use the Green's formula to obtain
$$
\tilde v (x)=\int_{ B_1\setminus B_\rho}
G(x,y)\tilde f(y).
$$
Corrollary \ref{cor-App1-3} follows from Lemma \ref{lem-App1}.
\end{proof}

\begin{cor}\label{cor-App1-3X}  For any $0 < \rho_0<\rho_1<\rho_2<1$,
let
 $$
\left\{
\begin{array}{rl}
-\Delta \tilde v= \tilde f\ge & 0,\qquad B_1,\\
\tilde v\ge &0, \qquad \partial B_1.
\end{array}
\right.
$$
Then, for  some constants $C$ depending only on $n, \rho_0,
\rho_1, \rho_2$, 
$$
\inf_{ \rho_1\le |x|\le \rho_2 }
\tilde v(x)
\ge \frac 1{C}
 \int_{ |y|\le \rho_0}
\tilde f(y)dy.
$$
\end{cor}

\begin{proof} This follows from Corollary \ref{cor-App1-3} by sending $\rho \rightarrow 0$.
\end{proof}

\begin{cor}\label{cor-App1-4}
For $0<r<R/2$, let 
 $$
\left\{
\begin{array}{rl}
-\Delta  v=  f\ge & 0,\qquad B_R\setminus B_r,\\
 v\ge &0, \qquad \partial (B_R\setminus B_r).
\end{array}
\right.
$$
Then, for  any $\frac {2r}R
<\rho_0<\rho_2<\rho_2 < 1$,
there exist some constants $C, C'>1$ depending only on $n, \rho_0,
\rho_1, \rho_2$,
\begin{align*}
\inf_{ \rho_1 R\le |x|\le \rho_2 R }
 v(x)
 	&\ge 
\frac 1{CR^{n-2} }
 \int_{  r \le |y|\le \rho_0 R}
 \left(  1-\frac {r^{n-2}}  { |y|^{n-2} }\right)
 f(y)dy\\
 	&\ge \frac 1{C'R^{n-2} }
 \int_{ 2r\le |y|\le \rho_0 R}
 f(y)dy.
\end{align*}

\end{cor}

\begin{proof}
Performing a change of variables
$$
\tilde v(x):= R^{n-2}  v(Rx),
\qquad \tilde f(x):= R^n f(Rx), \rho= r/R,
$$
we obtain Corollary \ref{cor-App1-4} from
Corollary \ref{cor-App1-3}
\end{proof}

\section{A remark on viscosity solutions}\label{Sec:AppVS}

In this section we consider the convergence of viscosity solutions in a slightly more general context. Let $\RR^{n \times n}$, $Sym^{n \times n}$, $Sym^{n \times n}_+$ denote the set of $n \times n$ matrices, symmetric matrices, and positive definite symmetric matrices, respectively. Let $\mathcal{M} = \RR^{n \times n}$ or $\mathcal{M} = Sym^{n \times n}$. Let $\Omega \subset \RR^n$, $U \subset \mathcal{M}$ be open, $F \in C(U)$, $A \in C(\Omega \times \RR \times \RR^n \times Sym^{n \times n}; \mathcal{M})$ and consider partial differential equations for the form
\[
F(A(x,u,\nabla u, \nabla^2 u)) = 0.
\]
To keep the notation simple, we will abbreviate $A[u] = A( \cdot,u,\nabla u, \nabla^2 u)$, and whenever we write $F(M)$, we implicitly assume that $M \in U$.

In applications, it is frequently assumed that
\begin{align}
&M + N \in U \text{ for } M \in U \text{ and } N \in Sym_+^{n\times n},\label{Eq:UMono}\\
&F(M + N) \geq F(M) \text{ for } M \in U \text{ and } N \in Sym_+^{n\times n},\label{Eq:FMono}\\
&A(x,z,p,M)- A(x,z,p,M + N)  \text{ is non-negative definite for } N \in Sym_+^{n \times n}.
	\label{Eq:AMono}
\end{align}

In the main body of the paper,  
\begin{equation}
\parbox{.8\textwidth}{$\mathcal{M} = Sym^{n \times n}$, $U$ is the set of symmetric matrices such that $\lambda(U) \in \Gamma$, $F(M) = f(\lambda(M)) - 1$, and $A[v] = A^v$.}
	\label{Eq:14IV2016-X1}
\end{equation}

The following definition is ``consistent'' with the assumptions \eqref{Eq:UMono}, \eqref{Eq:FMono} and \eqref{Eq:AMono} and with Definition \ref{Def:fVisSol}.

\begin{Def} A continuous function
 $v$ is an open set $\Omega\subset \RR^n$
is a viscosity supersolution (respectively, subsolution) of 
$$
F(A[v])=0, \ \ \mbox{in}\ \Omega,
$$
when the folowing holds: if $x_0\in \Omega$,
$\varphi\in C^2(\Omega)$, $(v-\varphi)(x_0)=0$,
and $v-\varphi\ge 0$ near $x_0$, then
$$
F(A[\varphi](x_0))\ge 0.
$$
(respectively, if $(v-\varphi)(x_0)=0$,
and $v-\varphi\le 0$ near $x_0$, then either 
$A[\varphi](x_0) \notin U$
or
$
F(A[\varphi](x_0))\le 0
$).  Equivalently, we write $F(A[v]) \geq 0$ (respectively,  $F(A[v]) \leq 0$) in $\Omega$ in the viscosity sense.

We say that $v$ is a viscosity solution if it is both
a viscosity supersolution and a viscosity subsolution.
\end{Def}

\begin{prop}\label{Prop:VisConv}
Let $\Omega \subset \RR^n$, $U \subset \mathcal{M}$ be open, $F \in C(U)$, $A \in C(\Omega \times \RR \times \RR^n \times Sym^{n \times n}; \mathcal{M})$ and assume that the structural conditions \eqref{Eq:UMono}, \eqref{Eq:FMono} and \eqref{Eq:AMono} are in effect.
\begin{enumerate}[(a)]
\item If $v_k$ satisfies $F(A[v_k]) \leq 0$ in $\Omega$ in the viscosity sense and if $v_k$ converges in $C^0_{loc}(\Omega)$ to $v$, then $v$ also satisfies $F(A[v]) \leq 0$ in $\Omega$ in the viscosity sense.
\item  Assume further that the closure of $F^{-1}([0,\infty))$ in $\mathcal{M}$ is a subset of $U$, i.e.
\begin{equation}
\text{ if $A_k \in U$, $F(A_k) \geq 0$ and $A_k \rightarrow A$ as $k \rightarrow \infty$, then $A \in U$.}
	\label{Eq:SuperClosed}
\end{equation}
If $v_k$ satisfies $F(A[v_k]) \geq 0$ in $\Omega$ in the viscosity sense and if $v_k$ converges in $C^0_{loc}(\Omega)$ to $v$, then $v$ also satisfies $F(A[v]) \geq 0$ in $\Omega$ in the viscosity sense.
\end{enumerate}
\end{prop}

\begin{rem}
In general, condition \eqref{Eq:SuperClosed} cannot be dropped in (b).

As a first example, consider
\begin{align*}
F(M) &= tr(M), \qquad U = \{M \in Sym^{n\times n}: tr(M) > 0\},\\
A[v] &= - \nabla^2 v.
\end{align*}
Clearly, \eqref{Eq:UMono}, \eqref{Eq:FMono} and \eqref{Eq:AMono} are satisfied. Now, if $v_k =  -\frac{1}{k}|x|^2$, then $A[v_k] \in U$ but $v = \lim_{k \rightarrow \infty} v_k = 0$ does not satisfy $A[v] \in U$.

The situation does not improve even if one imposes that $v_k$ is a solution and that $U$ is a maximal set where the ellipticity condition \eqref{Eq:FMono} holds. See Remark \ref{Rem:ASD}. 
\end{rem}

\begin{rem}\label{Cor:VCConf}
If $(f,\Gamma)$ satisfies \eqref{01}-\eqref{04} and if $(F,U)$ are given by \eqref{Eq:14IV2016-X1}, then all hypotheses of Proposition \ref{Prop:VisConv} are met. In particular if $f(\lambda(A^{v_k})) = 1$ in some open set $\Omega \subset \RR^n$ in the viscosity sense and $v_k$ converges in $C^0_{loc}(\Omega)$ to $v$, then $f(\lambda(A^v)) = 1$ in $\Omega$ in the viscosity sense.
\end{rem}

\begin{proof} The proof is standard and we include it here for readers' convenience. We will only show (b). The proof of (a) is similar. Fix $x_0 \in \Omega$ and assume that $\varphi \in C^2(\Omega)$ such that $(v - \varphi)(x_0) = 0$ and $v -\varphi \geq 0$ in some small ball $B_\rho(x_0)$. We need to show that $F(A[\varphi](x_0)) \geq 0$. 

For $\delta \in (0,\rho)$, let $\varphi_\delta(x) = \varphi(x) - \delta|x - x_0|^2$. Since $\varphi_\delta(x) = \varphi(x) -\delta^3 \leq v(x) - \delta^3$ for $|x - x_0| = \delta$, the convergence of $v_k$ to $v$ implies that
\[
v_k(x) \geq \varphi_\delta(x) + \frac{1}{2}\delta^3 \text{ for all } |x - x_0| = \delta \text{ and all sufficiently large } k.
\]
Thus, as $\min_{B_\delta(x_0)} (v - \varphi_\delta) = 0$, there is some $x_k \in B_\delta(x_0)$ such that
\[
\beta_k := (v_k - \varphi_\delta)(x_k) = \min_{\bar B_\delta(x_0)} (v_k - \varphi_\delta) \rightarrow 0 \text{ as } k \rightarrow \infty.
\]
Also, since
\[
\beta_k = (v_k - \varphi)(x_k) + \delta |x_k - x_0|^2 \geq (v_k - v)(x_k) + \delta|x_k - x_0|^2,
\]
we have that $x_k \rightarrow x_0$ as $k \rightarrow \infty$. 

Now since $\hat\varphi_k := \varphi_\delta + \beta_k$ satisfies $(v_k - \hat\varphi_k)(x_k) = 0$, $v_k - \hat \varphi_k \geq 0$ in $B_\delta(x_0)$, and since $v_k$ is a solution of $F(A[v_k]) = 0$, we infer that
\[
F(A[\hat\varphi_k](x_k)) \geq 0.
\]
Sending $k \rightarrow \infty$ and then $\delta \rightarrow 0$ and using \eqref{Eq:SuperClosed}, we deduce that $A[\varphi](x_0) \in U$ and
\[
F(A[\varphi](x_0)) \geq 0,
\]
as desired. (Note that \eqref{Eq:SuperClosed} is not required to show that (a).)
\end{proof}

\section{A calculus lemma}\label{Sec:AppC}

We collect here a couple calculus statements which was used when we applied the method of moving spheres in the body of the paper.

\begin{lem}
Let $f\in C^0(\RR^n \setminus \{0\})$, $n\ge 1$, $\nu>0$.  Assume that, for all $x \in \RR^n$ and $0 < \lambda \leq |x|$, there holds
$$
\Big( \frac \lambda {|y-x|}\Big)^\nu
f\Big(x+ \frac {  \lambda^2 (y-x) }{  |y-x|^2 }\Big) 
\le f(y) \text{ for } y \in \RR^n \setminus (B_\lambda(x) \cup \{0\}).
$$
Then $f$ is radially symmetric about the origin and $f: (0,\infty) \rightarrow (0,\infty)$ is non-increasing.
\label{lem-app1Sing}
\end{lem}

\begin{proof} This was established in \cite{Li06-JFA}. See the argument following equation (110) therein.
\end{proof}

\begin{cor}
Let $f\in C^0(\RR^n)$, $n\ge 1$, $\nu>0$.  Assume that
$$
\Big( \frac \lambda {|y-x|}\Big)^\nu
f\Big(x+ \frac {  \lambda^2 (y-x) }{  |y-x|^2 }\Big) 
\le f(y),\quad
\forall\ \lambda>0, \ x\in \RR^n,\ 
|y-x|\ge \lambda.
$$
Then $f\equiv $\ constant.
\label{lem-app1}
\end{cor}

\newcommand{\noopsort}[1]{}

\end{document}